\documentclass[preprint,11pt]{format}

\usepackage{amssymb,verbatim}
\usepackage{hyperref}
\usepackage{graphicx,color,amsmath,amsthm,amsopn,amstext,amscd,amsfonts,amssymb}
\usepackage{a4wide,macros_perso,url,xspace,bbm}

\def\x{{\textbf{\textit x\hspace{0.035cm}}}}
\def\y{{\textbf{\textit y\hspace{0.035cm}}}}

\def\vv{{\textbf{\textit v\hspace{0.035cm}}}}
\def\cC{\mathcal{C}}
\def\cDd{{\cal D}_{\#}}

\def\Had{H_{\alpha,\#}^1}
\def\Hmad{H_{\alpha,\#}^m}
\def\cSd{\mathcal{S}_{\#}}
\def\<{\langle}
\def\>{\rangle}

\begin{document}

\begin{frontmatter}



\title{Weighted Sobolev spaces for the Laplace equation \\in periodic infinite strips}


\author{Vuk Mili{{\v s}}i{\'c}}
\address{Universit{\'e} Paris 13,\\
Laboratoire Analyse, G{\'e}om{\'e}trie et Applications \\
CNRS UMR 7539
%99 avenue J.B. Cl{\'e}ment
%93430 Villetaneuse
FRANCE}

\author{Ulrich Razafison}
\address{Universit{\'e} de Franche-Comt{\'e},\\
Laboratoire de Math{\'e}matiques de Besan\c{c}on,\\
CNRS UMR 6623,
%16 route de Gray,
%25030 Besan\c{c}on Cedex,
FRANCE}

\begin{abstract}
This paper establishes isomorphisms for the Laplace operator in weighted Sobolev spaces (WSS). 
These $\wsdsnp{m}{\alpha}$-spaces are similar to standard Sobolev spaces $H^{m}(\R^n)$, 
but they are endowed with weights $(1+|x|^2)^{\alpha/2}$ prescribing  functions' growth or decay at infinity. Although well established in $\RR^n$ \cite{AmGiGiI.94}, these weighted results do not apply in the specific hypothesis of periodicity. 
This kind of problem appears when studying singularly perturbed domains (roughness, sieves, porous media, etc). 
When zooming on a single perturbation pattern, one often ends with a periodic problem set on an infinite strip. 
We present a unified framework that enables a systematic treatment of such problems. 
We  provide existence and uniqueness of solutions in our WSS. 
This gives a refined description of solution{\~O}s behavior at infinity which is of importance in the mutli-scale context.
These isomorphism results hold for any 
weight exponent $\alpha$ and any regularity index $m$. We then identify these solutions with the convolution of a 
Green function (specific to periodical infinite strips) and the given data. This identification is valid 
again for any $\alpha$ and any $m$ modulo some harmonic polynomials.
%In a next series of papers we shall uses results presented hereafter in order to solve the Stokes equations.
%We provide a complete framework in order to fill this gap. 
%The weights, which arise naturally from Hardy's  inequalities, permit to prove fundamental Poincar{\'e} inequalities relating functions' norms to their derivatives. 
%The results presented here  provide a unified framework for  microscopic problems in the boundary layer theory \cite{JaMiJDE.01, MiStent}. 
%We give isomorphisms results between WSS for any type of behavior at infinity, and for any index of regularity $m$. 
%Moreover we construct $G$, a periodical Green function and we prove that any solution provided by latter isomorphisms is
%actually a convolution with $G$ up to some harmonic  polynomials.
\end{abstract}

\begin{keyword}


weighted Sobolev spaces, Hardy inequality, isomorphisms of the Laplace operator, periodic infinite strip, Green function, convolution
\end{keyword}
\end{frontmatter}

% \linenumbers


\section{Introduction}\label{intro}

In this article, we solve the Laplace equation in a 1-periodic infinite strip in two space dimensions:
\begin{equation}
\label{Laplace_dans_Z}
\Delta u = f,\quad \text{ in } Z := ]0,1[ \times \RR.
\end{equation}
As the domain is infinite in the vertical direction, 
one introduces  weighted Sobolev spaces describing
 the behavior at infinity of the solution. 
This behavior is related to weighted Sobolev properties of $f$.

The usual weights, when adapted to our  problem, are polynomial functions
at infinity and regular bounded functions in the neighborhood of the origin:
they are powers of $\rho(y_2):=(1+y_2^2)^{1/2}$ and, in some critical cases, higher order 
derivatives are completed by logarithmic 
functions ($\rho(y_2)^\alpha\log^\beta(1+\rho(y_2)^2)$).

The literature on the weighted Sobolev spaces is wide \cite{Ha.71,AmGiGiII.97,Girault_DIE_1994, Girault_JFSUT_1992,
Boulmezaoud_M2AS_2003, Boulmezaoud_M2AS_2002,MR2665021, MR2454803, MR2329265, MR2329264} and deals with 
various types of domains. To our knowledge, 
this type of weights has not been applied to problem \eqref{Laplace_dans_Z}. 
The choice of the physical domain comes from periodic singular problems : in \cite{JaMiJDE.01,NeNeMi.06,MiStent},
 a zoom around a domain's periodic $\e$-perturbation  leads to set an obstacle of size 1 in $Z$ and to consider 
a microscopic problem defined on {\em a boundary layer}. 
The behavior at infinity of this {\em microscopic} solution is of importance: it provides
an  averaged feed-back on the macroscopic scale (see \cite{BrMiQam} and references therein).
This paper is a first step towards a systematic
analysis of such microscopic problems. We intend to give a standard framework  to skip 
tedious and particular proofs related to the unboundedness of $Z$.

We provide  isomorphisms of the Laplace operator between our weighted Sobolev spaces.
It is the first step among results in the spirit of \cite{AmGiGiII.97,AmGiGiI.94,AmNe01}. Since
error estimates for boundary layer problems \cite{JaMiJDE.01, MiStent, NeNeMi.06} are mostly performed 
in the $H^s$ framework we focus here on weighted Sobolev spaces  $W^{m,p}_{\alpha,\beta}(Z)$  with $p=2$.
There are three types of tools used here: 
arguments specifically related to weighted Sobolev spaces \cite{AmGiGiI.94,These_Giroire},
variational techniques from the homogenization literature \cite{MiStent, Ba.Siam.76, JaMiFilter} 
and some potential theory methods \cite{McOwen.1979}.
A general scheme might illustrate how these ideas relate one to each other :
\begin{itemize}
\item[-] a Green function $G$ specific to the periodic infinite strip is exhibited for the Laplace operator. 
%The convolution with $G$ is identified modulo some polynomial 
%for a wide set of positive weights with the solution $u$ above in a specific range of weights.
The convolution of  $f$ with $G$ provides an explicit solution to \eqref{Laplace_dans_Z}. 
A particular attention is provided to give the weakest possible  meaning to the latter convolution 
under minimal requirements on $f$. 
\item[-] variational {\em inf-sup} techniques provide {\em a priori }  estimates, in an initial 
range of weights, for interior and exterior problems with Dirichlet boundary conditions. Then a lift is built
in order to cancel the lineic Dirac mass that's introduced by  artificial interfaces. This leads to first isomorphism results.
\item[-] these arguments are then applied to weighted derivatives and give natural regularity shift results in $H^m_{\alpha,\#}(Z)$ 
for $m\geq 2$ (see below).
\item[-] by duality and appropriate use of generalized Poincar{\'e} estimates (leading to interactions - orthogonality or quotient spaces - with various polynomial families), 
one ends  with generic isomorphism result that reads
\begin{thm}\label{thm.isom.final}
 For any   $\alpha\in \R$ and any $m \in \Z$, the mapping
$$
\Delta\,:X_{\alpha,\#}^{m+2}(Z)/\PP_{q(m+2,\alpha)}^{'\Delta}\mapsto X_{\alpha,\#}^m(Z)\bot\PP_{q(-m,-\alpha)}^{'\Delta}
$$
is an isomorphism.
\end{thm}
The spaces $\wsx{m}{\alpha}{Z}$ are generalized WSS whose precise definition is given in section \ref{sec.critic}.
They are introduced in order to deal with critical  values  of $(m,\alpha)$.
The spaces $\pP_{q(m,\alpha)}$ of harmonic polynomials included in $ \wsx{m}{\alpha}{Z}$ are defined in section \ref{sec.notations}.
% for which  logarithmic weights are required.
\item[-] for all  values of $\alpha\in \R$ and $m\in\Z$, we identify explicit solutions obtained via convolution 
with solutions given in Theorem \ref{thm.isom.final}. This gives our main result:

\begin{thm}\label{thm.final}
Let $m\in\Z$, $\alpha\in\R$ and $f\in X_{\alpha,\#}^{m}(Z)\bot\PP_{q(-m,-\alpha)}^{'\Delta}$.
Then $G * f\in X_{\alpha,\#}^{m+2}(Z)$ is the unique  solution of the Laplace equation \eqref{Laplace}
up to a polynomial of $\pP_{q(m+2,\alpha)}$.
Moreover, we have the estimate
$$\|G * f\|_{X_{\alpha,\#}^{m+2}(Z)/\pP_{q(m+2,\alpha)}}\leq\|f\|_{X_{\alpha,\#}^{m}(Z)}.$$
\end{thm}

\end{itemize}

The paper is organized as follows. In section \ref{sec.notations}, 
we define the basic functional framework and some preliminary results used in this article.
In section \ref{sec.poincare}, we adapt weighted Poincar{\'e} estimates to our setting. 
%Then we adapt Hardy type inequalities to the infinite periodic band in section \ref{sec.poincare}. 
Then (section \ref{sec.fourier}), we introduce a mixed Fourier transform: it is a  discrete Fourier transform in the horizontal 
direction and a continuous transform in the vertical direction. 
This operator allows the explicit computation of the Green function, 
and we derive weighted and standard estimates of the convolution with this fundamental solution. 
At this stage, we prove a series of isomorphisms in the non-critical 
case (section \ref{section_Laplace}) by variational techniques. 
In the last section, we extend these to the  critical cases. In a last step
we identify any of the solutions above with the convolution between $G$ and the data $f$,  this proves Theorem \ref{thm.final}.

\section{Notation, preliminary results and functional framework}\label{sec.notations}

\subsection{Notation and preliminaries}

We denote by $Z$ the two-dimensional infinite strip defined by
$$Z:=]0,1[\times\R.$$
We use bold characters for vector or matrix
fields. A point in $\R^2$ is denoted by $\y=(y_1,y_2)$ and its
distance to the origin by
$$r:=|\y|=\left(y_1^2+y_2^2\right)^{1/2}.$$
Let $\N$ denote the set of non-negative integers, $\Z$ the set
of all integers and $\Z^*=\Z\setminus\{0\}$. We denote by
$[k]$ the integer part of $k$. For any $j\in\Z$, $\PP_j'$ stands for
the polynomial space of degree less than or equal to $j$ that only depends on $y_2$. 
If $j$ is a negative integer, we set by convention $\PP_j'=\{0\}$. 
We define $\pP_j$ the subspace of harmonic polynomials of $\PP_j'$. 
The support of a function $\varphi$ is denoted by $\mbox{supp}(\varphi)$.
We recall that $\cD(\R)$ and $\cD(\R^2)$ are  spaces of $\cC^\infty$ functions with compact support in $\R$ and $\R^2$
respectively, $\cD'(\R)$ and $\cD'(\R^2)$ their dual spaces, namely the spaces of distributions. We denote by $\cS(\R)$ the Schwartz space of functions in $\cC^\infty(\R)$ with rapid decrease at infinity, and by $\cS'(\R)$ its dual, i.e.
the space of tempered distributions. We recall that, for $m\in\N$, $H^m$ is the classical Sobolev space and we denote by $H_\#^m(Z)$
the  space of functions that belong to $H^m(Z)$ and that are 1-periodic in the $y_1$ direction.
Given a Banach space $B$ with its dual $B'$ and a closed subspace $X$ of $B$, we denote by $B'\bot X$ the subspace of $B'$ orthogonal to $X$, \textit{i.e.}:
$$B'\bot X=\{f\in B',\,\,\forall v\in X, \<f,v\>=0\}=(B/X)'.$$

\noindent We introduce $\tau_\lambda$ the operator of translation of $\lambda\in\Z$ in the $y_1$ direction. 
If $\Phi:\R^2\mapsto\R$ is a function, then we have
$$(\tau_\lambda\Phi)(\y):=\Phi(y_1-\lambda,y_2).$$
%In this paper, we shall also apply the operator $\tau_\lambda$ to single-variable function.
%If $\Phi$ is a such function, it reduces to 
%$$(\tau_\lambda\Phi)(t):=\Phi(t-\lambda).$$
\noindent For any $\Phi\in\cD(\R^2)$, we set 
$
\overline{\omega}\Phi:=\sum_{\lambda\in\Z}\tau_\lambda\Phi
$,
which is the $y_1$-periodical transform of $\Phi$. 
The mapping $y_1\mapsto\overline{\omega}\Phi(\y)$ belongs to $\cC^\infty(\R)$ and is 1-periodic.
Observe that there exists a function $\theta$ satisfying %(see \cite{VoKhacKhoan.Book2.1972}, p. 63)
$$\theta\in\cD(\R)\mbox{ and }\overline{\omega}\theta=1.$$
More precisely, consider a function $\psi\in\D(\R)$ such that $\psi>0$ on the interior of its support. 
Then we simply set $\theta=\frac{\psi}{\overline{\omega}\psi}$ . 
The function $\theta$ is called a periodical $\cD(\R)$-partition of unity.\\

\noindent If $T\in\cD'(\R^2)$, then for all $\varphi\in\cD(\R^2)$, we set
$$\<\tau_\lambda T,\varphi\>:=\<T,\tau_{-\lambda}\varphi\>,\,\,\,\lambda\in\Z.$$

\noindent Similarly, if $T\in\cD'(\R^2)$ has a compact support in the $y_1$ direction, 
then the $y_1$-periodical transform of $T$, denoted by $\overline{\omega}T$, is defined by

$$\<\overline{\omega} T,\varphi\>:=\<T,\overline{\omega}\varphi\>,\hspace{3mm}\forall\varphi\in\cD(\R^2).$$

\noindent These definitions are well-known, we refer for instance to \cite{VoKhacKhoan.Book2.1972} and \cite{Book_Schwartz}.

\begin{rmk}
\label{rem_preliminaire}
\
{\rm
Let $\Phi$ be in $\cD(\R^2)$. Then we have
\begin{equation}
\label{rem_preliminaire_eq1}
\tau_\lambda\partial_i\Phi=\partial_i(\tau_\lambda\Phi),\hspace{3mm}\forall\lambda\in\Z,
\end{equation}
\begin{equation}
\label{rem_preliminaire_eq2}
\overline{\omega}(\partial_i\Phi)=\partial_i(\overline{\omega}\Phi)
\end{equation}
and
\begin{equation}
\label{rem_preliminaire_eq3}
\|\tau_\lambda\Phi\|_{L^2(Z)}\leq\|\Phi\|_{L^2(\R^2)}.
\end{equation}
\noindent As a consequence of equality \eqref{rem_preliminaire_eq2}, if $u\in\cD'(\R^2)$, we also have 
\begin{equation}
\label{rem_preliminaire_eq4}
\overline{\omega}(\partial_i u)=\partial_i(\overline{\omega}u).
\end{equation}
}
\end{rmk}
\medskip
\noindent The next lemma is used to prove the density result of Proposition \ref{prop_densite}.
\begin{lemma}
\label{lem_preliminaire}
Let $K$ be a compact of $\R^2$. Let $u$ be in $H^m(\R^2)$ and have a compact support included in $K$.
Then we have $$\|\overline{\omega}\,u\|_{H^m(Z)}\leq N(K)\|u\|_{H^m(\R^2)},$$
where $N(K)$ is an integer only depending on $K$.
\end{lemma}

\begin{proof}
Let us first notice that, since $K$ is compact, there is a finite number of $\lambda\in\Z$ such that 
$\mbox{supp}(\tau_\lambda u)\cap[0,1]\times\R$ is not an empty set. This number is bounded by a finite integer $N(K)$ that only depends on $K$. 
It follows that $\overline{\omega}\,u$ is a finite sum and 
$$\|\overline{\omega}\,u\|_{L^2(Z)}\leq N(K)\|u\|_{L^2(\R^2)}.$$
The end of the proof then follows from \eqref{rem_preliminaire_eq4}.
\end{proof}

\noindent We now define the following spaces 
$$
%\left\{
\begin{aligned}
\cC_{\#}^\infty(Z)& :=\{&\varphi:Z\mapsto\R,\,
y_1\mapsto\varphi(\y)\in\cC^\infty([0,1])\mbox{ 1-periodic, }%\nonumber\\
& y_2\mapsto\varphi(\y)\in\cC^\infty(\R)\},%\nonumber
\\
\cDd(Z)&:=\{&\varphi:Z\mapsto\R,\,
y_1\mapsto\varphi(\y)\in\cC^\infty([0,1])\mbox{ 1-periodic, }%\nonumber\\
& y_2\mapsto\varphi(\y)\in\cD(\R)\},%\nonumber
\\
\cSd(Z)&:=\{&\varphi:Z\mapsto\R,\,
y_1\mapsto\varphi(\y)\in\cC^\infty([0,1])\mbox{ 1-periodic, }%\nonumber\\
& y_2\mapsto\varphi(\y)\in\cS(\R)\}.%\nonumber
\end{aligned}
%\right.
$$
\noindent The dual spaces of $\cDd(Z)$ and $\cSd(Z)$ are denoted by $\cDd'(Z)$ and $\cSd'(Z)$ respectively.\\ 

%\noindent Let now $T\in\cD'(\R)$ be a  
 
\subsection{Weighted Sobolev spaces in an infinite strip}
\label{weight}

\noindent We introduce the weight 
$$
\rho\equiv\rho(y_2):=\left(1+y_2^2\right)^{1/2}.
$$
Observe that, for $\lambda\in\N$ and for $\gamma\in\R$, as $|y_2|$ tends to infinity, we have
\begin{equation}
\label{derivation_poids}
\left|\frac{\partial^\lambda\rho^{\gamma}}{\partial y_2^\lambda}\right|\leq C\rho^{\gamma-\lambda}.
\end{equation}
For $\alpha\in\R$, we define the  weighted space
$$L_{\alpha}^2(Z)\equiv H_{\alpha,\#}^0(Z)=\{u\in\cD'(Z),\,\rho^{\alpha}u\in L^2(Z)\},$$
\noindent which is a Banach space equipped with the norm
$$\|u\|_{L_{\alpha}^2(Z)}=\|\rho^{\alpha}u\|_{L^2(Z)}.$$

\begin{proposition}
\label{densite_L2}
The space $\cD(Z)$ is dense in $L_{\alpha}^2(Z)$.
\end{proposition}

\begin{proof}
Let $u$ be in $L_{\alpha}^2(Z)$. Then, by definition of the space $L_{\alpha}^2(Z)$,
it follows that $\rho^\alpha u\in L^2(Z)$. Therefore there exists a sequence $(u_n)_{n\in\N}\subset\cD(Z)$
such that $u_n$ converges to $\rho^\alpha u$ in $L^2(Z)$ as $n\to\infty$. Thus, setting $v_n=\rho^{-\alpha}u_n$, 
we see that $v_n$ converges to $u$ in $L_\alpha^2(Z)$ as $n\to\infty$.
\end{proof}

\begin{rmk}
{\rm Observe that we have the algebraic inclusion $\cD(Z)\subset\cDd(Z)$.
It follows that $\mbox{the space }\cDd(Z)\mbox{ is dense in }L_\alpha^2(Z).$}
\end{rmk}
For a non-negative integer $m$ and a real number $\alpha$, we set
$$
k=k(m,\alpha):=\left\{\begin{array}{ll}
\vspace{2mm}
-1\hspace{15mm}\mbox{ if }\hspace{1.5mm}\,\,\,\,\alpha\notin\{1/2,...,m-1/2\}\\
m-1/2-\alpha\hspace{2mm}\mbox{ if
}\hspace{1.5mm}\alpha\in\{1/2,...,m-1/2\}\end{array}\right.
$$
and we introduce the weighted Sobolev space
\begin{align}
\Hmad(Z):=\{u\in\cDd'(Z);\forall\lambda\in\N^2,%\nonumber\\
&0\leq|\lambda|\leq k,\hspace{1.5mm}\rho^{-m+|\lambda|}(\ln(1+\rho^2))^{-1}\partial^{\lambda}u\in L_\alpha^2(Z),\nonumber\\
&k+1\leq|\lambda|\leq
m,\hspace{1.5mm}\rho^{-m+|\lambda|}\partial^{\lambda}u\in
L_\alpha^2(Z)\},\nonumber
\end{align}
which is a Banach space when endowed with the norm
\begin{align}
\|u\|_{\Hmad(Z)}:=&\left(\displaystyle\sum_{0\leq|\lambda|\leq
k}
\|\rho^{-m+|\lambda|}(\ln(1+\rho^2))^{-1}\partial^{\lambda}u\|_{L_\alpha^2(Z)}^2\right.%\nonumber\\
%&
\left.+\displaystyle\sum_{k+1\leq|\lambda|\leq m}
\|\rho^{-m+|\lambda|}\partial^{\lambda}u\|_{L_\alpha^2(Z)}^2\right)^{1/2}.\nonumber
\end{align}
\noindent We define the semi-norm
$$|u|_{\Hmad(Z)}:=\left(\sum_{|\lambda|=m}\|
\partial^{\lambda}u\|_{L_\alpha^2(Z)}^2\right)^{1/2}.$$ 
\noindent Observe that the logarithmic weight function only appears for the so-called critical cases $\alpha\in\{1/2,...,m-1/2\}$. 
Local properties of the spaces $\Hmad(Z)$ coincide with those 
of the classical Sobolev space $H_\#^m(Z)$. When $\alpha\notin\{1/2,...,m-1/2\}$
we have the following algebraic and topological inclusions~:
\begin{equation}
\label{inclusions_poids}
\Hmad(Z)\subset H_{\alpha-1,\#}^{m-1}(Z)\subset...\subset L_{\alpha-m}^2(Z).
\end{equation}
When $\alpha\in\{1/2,...,m-1/2\}$, the logarithmic weight function appears, so that we only have the inclusions 
\begin{equation}
\label{inclusions_poids_log}
\Hmad(Z)\subset...\subset H_{1/2,\#}^{m-\alpha+1/2}(Z).
\end{equation}
Observe that the mapping 
\begin{equation}
\label{derivation}
u\in\Hmad(Z)\mapsto\partial^\lambda u\in H_{\alpha,\#}^{m-|\lambda|}(Z)
\end{equation}
is continuous for $\lambda\in\N^2$.
Using \eqref{derivation_poids}, for $\alpha,\gamma\in\R$
such that $\alpha\notin\{1/2,...,m-1/2\}$ and $\alpha-\gamma\notin\{1/2,...,m-1/2\}$, the mapping 
\begin{equation}
\label{multiplication_poids}
u\in\Hmad(Z)\mapsto\rho^{\gamma}u\in H_{\alpha-\gamma,\#}^m(Z)
\end{equation}
is an isomorphism.\\
Let $q$ be  the greatest non-negative integer such that $y_2^q \in \Hmad(Z)$. An easy computation 
shows that $q$ reads~: 
\begin{equation}
\label{def_q} 
q\equiv q(m,\alpha):=\left\{\begin{array}{ll}
m-3/2-\alpha,\,\,\,\mbox{if }\alpha+1/2\in\{i\in\Z;i\leq0\},\\

[m-1/2-\alpha],\,\,\,\mbox{ otherwise.}
\end{array}\right.
\end{equation}
\noindent For example (see also tab. \ref{tab.poly}),  if $m=0$ one has :
\begin{equation}
\label{def_q_bis} 
q=q(0,\alpha):=\left\{\begin{array}{ll}
-3/2-\alpha,\,\,\,\mbox{if }\alpha=\pm(\frac{1}{2}+i),  i\in\Z^*,\\

[-1/2-\alpha],\,\,\,\mbox{ otherwise.}
\end{array}\right.
\end{equation}

\begin{table}[ht!]
\begin{center}
\begin{tabular}{r|c|c|c|c|}
$m$ $\backslash$ $\alpha$ & $\left[ -\frac{5}{2}; -\td\right[$ & $\left[ -\td;-\ud \right[$ & $\left[ -\ud;\ud \right[$  & $\left[\ud; \td \right[$ \\
\hline
$\PP'_{q(0,\alpha)}$ & $\PP_1'$  & $\PP_0'$  & 0 & 0 \\
\hline
$\PP'_{q(1,\alpha)}$ & $\PP'_2$ & $\PP'_1$ & $\PP'_0$ & 0 \\
\hline
$\PP'_{q(2,\alpha)}$ & $\PP'_3$ & $\PP'_2$ & $\PP'_1$ & $\PP'_0$  \\
\hline
\end{tabular}
\caption{Polynomial spaces included in $\wsd{m}{\alpha}{Z}$ for various values of $\alpha$ and $m$}\label{tab.poly}
\end{center}
\end{table}
\begin{proposition}
\label{prop_densite}
The space $\cDd(Z)$ is dense in $\Hmad(Z)$.
\end{proposition}

\begin{proof}
\noindent The idea of the proof comes from \cite{AmGiGiI.94} and \cite{Rapport_Nguetseng}. Let $u$ be in $\Hmad(Z)$.
\begin{enumerate}[(i)]
\item  We first approximate $u$ by functions with compact support in the $y_2$ direction.
Let $\Phi\in\cC^\infty([0,\infty[)$ such that $\Phi(t)=0$ for $0\leq t\leq1$, $0\leq\Phi(t)\leq1$,
for $1\leq t\leq2$ and $\Phi(t)=1$, for $t\geq2$. For $\ell\in\N$, we introduce the function $\Phi_\ell$, defined by
\begin{equation}
\label{def_phi}
\Phi_\ell(t)=\left\{\begin{aligned}
&\displaystyle\Phi\left(\frac{\ell}{\ln t}\right)& \text{ if }t>1,\\
&1&\mbox{ otherwise}.
\end{aligned}\right.
\end{equation}
Note that we have $\Phi_\ell(t)=1$ if $0\leq t\leq e^{\ell/2}$, $0\leq\Phi_\ell(t)\leq1$ if $e^{\ell/2}\leq t\leq e^\ell$
and $\Phi_\ell(t)=0$ if $t\geq e^\ell$. Moreover, for all $\ell\geq2$, $t\in[e^{\ell/2},e^\ell]$, $\lambda\in\N$, 
owing that $t\leq\sqrt{1+t^2}\leq\sqrt{2}\,t$ and $\ln t\leq\ln(2+t^2)\leq3\ln t$, we prove that (see \cite{AmGiGiI.94}, Lemma 7.1): 
$$\left|\frac{d^\lambda}{dt^\lambda}\Phi\left(\frac{\ell}{\ln t}\right)\right|\leq \frac{C}{(1+t^2)^{\lambda/2}\ln(2+t^2)},$$
\noindent where $C$ is a constant independent of $\ell$. We set $u_\ell(\y)=u(\y)\Phi_\ell(y_2)$. Then, proceeding as in 
\cite{AmGiGiI.94} (Theorem 7.2), one checks easily that 
$u_\ell$ belongs to $\Hmad(Z)$, has a compact support in the $y_2$ direction, and that $u_\ell$ converges 
to $u$ in $\Hmad(Z)$ as $\ell$ tends to $\infty$.\\ 
Thus, the functions of $\Hmad(Z)$ with compact support in the $y_2$ direction are dense in $\Hmad(Z)$ and, 
we may assume that $u$ has a compact support in the $y_2$ direction.

\item Let $\theta$ be a periodical $\cD(\R)$-partition of unity and let $(\alpha_j)_{j\in\N}$ 
be a sequence such that $\alpha_j\in\cD(\R^2)$, $\alpha_j\geq0$, $\int_{\R^2}\alpha_j(\x)d\x=1$ and the support of $\alpha_j$ is 
included in the closed ball of radius $r_j>0$ and centered at $(0,0)$ where $r_j\to0$ as $j\to\infty$. It is well known that 
as $j\to\infty$, $\alpha_j$ converges in the distributional sense to the Dirac measure. We set
$w(\y)=\theta(y_1)u(\y)$. Then $w$ belongs to $H^m(\R^2)$ and has a compact support. 
Moreover, since $\overline{\omega}\theta=1$, we have $\overline{\omega}w=\overline{\omega}(\theta u)=(\overline{\omega}\theta)u=u$. We define 
$\varphi_j=w*\alpha_j$. Then $\varphi_j$ belongs to $\cD(\R^2)$ and converges to $w$ in $H^m(\R^2)$ as $j$ tends to $\infty$.
Let $\psi_j=\overline{\omega}\varphi_j$, then $\psi_j$ belongs to $\cDd(Z)$ and thanks to Lemma \ref{lem_preliminaire},
$\psi_j$ converges to $\overline{\omega}\,w=u$ in $\Hmad(Z)$ as $j$ tends to~$\infty$.    
\end{enumerate}
\end{proof}
\medskip
\noindent The above proposition implies that the dual space of $\Hmad(Z)$ denoted by $H_{-\alpha,\#}^{-m}(Z)$ is 
a subspace of $\cDd'(Z)$.

\section{Weighted Poincar{\'e} estimates}\label{sec.poincare}

\noindent Let $R$ be a positive real number. For $\beta\neq-1$, one uses the standard Hardy estimates:
\begin{equation}
\label{eq.hardy}
\int_R^\infty |f(r)|^2 r^\beta dr \leq \left( \frac{2}{\beta +1} \right)^2 \int_R^\infty |f'(r)|^2 r^{\beta+2} dr,
\end{equation}
while for the specific case when $\beta=-1$, one switches to 
$$
\int_R^\infty \frac{|f(r)|^2}{(\ln r )^2\,r} dr \leq \left(\frac{4}{3}\right)^2  \int_R^\infty |f'(r)|^2\,r dr.
$$

\noindent We now introduce the truncated domain $$Z_R:= ]0,1[\times (]-\infty ,-R[\cup ]R,+\infty[).$$

\noindent Using the above Hardy inequalities in the $y_2$ direction, we can easily prove the following lemma.

\begin{lem}\label{lemme_3.2}
Let $\alpha$, $R>1$ be real numbers and let $m\geq1$ be an integer. Then there exists a constant $C_\alpha^m$ such that
\begin{equation}
\label{Poincare_y2_ZR}
\forall\varphi\in\cDd(Z_R),\hspace{3mm} \nrm{\varphi}{\Hmad(Z_R)} \leq C_\alpha^m  \nrm{\partial_2\varphi}{\wsd{m-1}{\alpha}{Z_R}}. %\|\partial_2\varphi\|_{H_{\alpha}^{m-1}(Z_R)}.
\end{equation}
As a consequence, we also have 
\begin{equation}
\label{Poincare_ZR}
\forall\varphi\in\cDd(Z_R),\hspace{3mm} \nrm{\varphi}{\Hmad(Z_R)} \leq C_\alpha^m|\varphi|_{\Hmad(Z_R)}.
\end{equation}
\end{lem}

\noindent For the particular case when $\alpha\in\{1/2,...,m-1/2\}$, \eqref{Poincare_ZR} cannot hold without introducing logarithmic weights in the definition of the space $\wsd{m}{\alpha}{Z}$.

%\noindent Observe that, Inequality \eqref{Poincare_ZR} is the reason of introducing the logarithmic weight in the definition of the space $\Hmad(Z)$ for the particular case $\alpha\in\{1/2,...,m-1/2\}$.\\

\noindent Proceeding as in \cite{AmGiGiI.94} (Theorem 8.3), we have the  Poincar{\'e}-type inequalities :
\begin{thm}
\label{thm_Poincare}
Let $\alpha$ be a real number and $m\geq1$ an integer. Let $j:=\min(q(m,\alpha),m-1)$
where $q(m,\alpha)$ is defined by \eqref{def_q}. 
Then there exists a constant $C>0$, such that for any $u\in\Hmad(Z)$, we have 
%\begin{equation}
%\label{Poincare_y2}
%\inf_{v\in V_\alpha^{y_2}(Z)}\|u+v\|_{\Hmad(Z)}\leq C\|\partial_2 u\|_{H_{\alpha,\#}^{m-1}(Z)}.
%\end{equation}
%Moreover,  there also exists a constant $C>0$, such that for any $u\in\Hmad(Z)$,
\begin{equation}
\label{Poincare}
\inf_{\lambda\in\PP_j'}\|u+\lambda\|_{\Hmad(Z)}\leq C|u|_{\Hmad(Z)}.
\end{equation}
In other words, the semi-norm $|.|_{\Hmad(Z)}$ defines on $\Hmad(Z)/\PP_j'$
a norm which is equivalent to the quotient norm. 
\end{thm}

\begin{comment}
\noindent In order to present the last results of the section, we introduce the space:
$$V_\alpha(Z)=\{\vv\in L_{\alpha}^2(Z),\,\,\,\dive\vv=0\}.$$
\noindent Then, as a consequence of Theorem \ref{thm_Poincare}, we have the following isomorphism
results on the gradient and divergence operators, whose proof is similar to the proof of Proposition 4.1
in \cite{AmGiGiI.94}.

\begin{proposition}
\label{prop.grad.div}
Assume that $\alpha\in\R$. 
\begin{enumerate}
\item If $\alpha>1/2$, then 
\begin{itemize}
\item the  gradient operator is an isomorphism:
$$\nabla\,:\,\Had(Z)\mapsto L_\alpha^2(Z)\bot V_{-\alpha}(Z),$$
\item the  divergence operator is an isomorphism:
$$\dive\,:L_{-\alpha}^2(Z)/V_{-\alpha}(Z)\,\mapsto H_{-\alpha,\#}^{-1}(Z).$$
\end{itemize}
\item If $\alpha\leq1/2$, then 
\begin{itemize}
\item the  gradient operator is an isomorphism:
$$\nabla\,:\,\Had(Z)/\R\mapsto L_\alpha^2(Z)\bot V_{-\alpha}(Z),$$
\item the  divergence operator is an isomorphism:
$$\dive\,:L_{-\alpha}^2(Z)/V_{-\alpha}(Z)\,\mapsto H_{-\alpha,\#}^{-1}(Z)\bot\R.$$
\end{itemize}
\end{enumerate}
\end{proposition}
\end{comment}
\renewcommand\Re{\operatorname{\mathfrak{Re}}}
\renewcommand\Im{\operatorname{\mathfrak{Im}}}

\section{The mixed Fourier transform (MFT) and the Green function}\label{sec.fourier}

\noindent The purpose of this section is twofold: \begin{enumerate}[(i)]
\item we look for the fundamental solution for the Laplace equation in $Z$,
\item we estimate the convolution with this solution in weighted Sobolev spaces. 
\end{enumerate}
To achieve this goal, we define the MFT and an adequate functional setting. 

\subsection{Rapidly decreasing functions}
Let set $\Gamma:=\Z\times\RR$ and let us write the locally convex linear topological spaces ~:
$$
\begin{aligned}
\tS(\Gamma) := &\left\{ \ti{\varphi} : \Z \times \RR \to \RR  \; \text{ s.t. } \forall k \in \Z \quad  \ti{\varphi}(k,\cdot) \in \cS (\R) \vphantom{ \int_\R^{100} } \right. \\
& \left. \text{ and } \sup_{ k \in \Z, l \in \R } \left| k^\alpha l^\beta \partial^\gamma_{l^\gamma} \ti{\varphi}(k,l) \right| < \infty, \, \forall (\alpha,\beta,\gamma) \in \N^3 \right\}.
\end{aligned}
$$ 
The space $\tS(\Gamma)$ is endowed with the semi-norms
$$
\snrm{ \ti{\varphi}}{\alpha,\beta,\gamma} := \sup_{ k \in \Z, l \in \R } \left| k^{\alpha'} l^{\beta'} \partial^{\gamma'}_{l^{\gamma'}} \ti{\varphi} \right| ,\quad \forall\alpha'\leq\alpha,\,\beta'\leq\beta,\,\gamma'\leq\gamma.
$$
We define also 
$$
l^2(\Z;L^2(\RR)) := \left\{ u \in \tS'(\Gamma) \text{ s.t. }\sum_{k \in \Z} |u(k,\cdot)|^2_{L^2(\RR)} < \infty  \right\} .
$$
\begin{proposition}\label{prop.density} $\tS(\Gamma)$ is dense in $l^2(\Z;L^2(\RR))$.
%\end{enumerate}
\end{proposition}

\begin{proof}
As $u \in l^2(\Z;L^2(\RR))$ 
$$
\forall \e > 0 \quad \exists k_0 > 0 \text{ s.t. } \sum_{|k| > k_0} \nrm{u(k,\cdot)}{L^2(\RR)}^2 < \frac{\e}{2}.
$$
We set
$$
u^\delta := \begin{cases}
u^\delta (k,\cdot) &  \text{ if } |k| < k_0, \\
0&  \text{ otherwise } ,
\end{cases}
$$ 
where $u^\delta (k,\cdot)$ is  a standard smooth approximation  of $u(k,\cdot)$ in $\D(\R)$.
Then, we choose $\delta(k_0)$ s.t. 
$$
\forall k \st |k| \leq k_0 \quad \nrm{u(k,\cdot)-u^\delta (k,\cdot)}{L^2(\RR)}^2 < \frac{\e}{2 k_0},
$$
which finally gives
$$
\nrm{u-u^\delta}{ l^2(\Z;L^2(\RR))}^2 \leq \sum_{|k|\leq k_0} \nrm{u(k,\cdot)-u^\delta (k,\cdot)}{L^2(\RR)}^2  
+ \sum_{|k| > k_0} \nrm{u(k,\cdot)}{L^2(\RR)}^2< \e
$$
which proves the claim, since $u^\delta$ trivially belongs to $\tS(\Gamma)$.
\end{proof}

\begin{defi}\label{def.ff} 
We define the MFT operator $\cF:\cSd(Z) \to \tS(\Gamma)$ as
$$
\cF(\varphi)(k,l) := \int_{Z} \varphi ( \y ) e^{-i (\ti{k} y_1 + \ti{l} y_2 )} d\y, \quad \forall (k,l) \in \Gamma,
$$
where $\ti{k}:= 2 \pi k$, and the same holds for $\ti{l}$.
\end{defi}

We list below some basic tools  concerning the MFT 
needed in the remaining of the paper. 
They can be proved following classical
arguments (\cite{Book.Yos}, \cite{Rudin.Book.1975}, \cite{Book.ReSi}).

\begin{proposition}\label{prop.fourier.s}
\begin{enumerate}[(i)]
\item  The operator $\cF:\cSd(Z) \to \tS(\Gamma)$ is an isomorphism and the inverse operator is explicit:
$$
\cF^{-1} (\ti{\varphi})(\y) = \sum_{k\in \Z} \int_{\RR} \ti{\varphi}(k,l) e^{ i (\ti{k} y_1+ \ti{l} y_2)} dl, \quad \forall \ti{\varphi} \in \tS(\Gamma),{\nobreakspace}\quad \forall \y \in Z .
$$
\item 	One has in a classical way, for any $(f,g) \in \cSd(Z)\times\cSd(Z)$, that
$$
%\begin{aligned}
\int_{Z} f \cdot \ov{g}\,d \y %&
= \sum_k \int_{\RR} \cF(f)(k,l) \cdot \ov{\cF(g)}(k,l) dl ,\text{ and }%\quad %\\
 \cF( f * g ) %&
 = \cF(f) \cF(g), \quad  \forall (k,l) \in \Gamma.
%\end{aligned}
$$
\end{enumerate}

\end{proposition}

\subsection{Tempered distributions and MFT}

In a quite natural manner we extend the concepts introduced above to  tempered distributions.
\begin{defi}\label{defi.fourier.distrib}
\begin{enumerate}[(i)]
\item A linear form $\ti{T}$ acting on $\tS(\Gamma)$ is s.t. $\exists (\alpha ,\beta ,\gamma ) \in \N^3$ 
$$
\left| \ti{T} (\ti{\varphi}) \right| \leq C \snrm{ \ti{\varphi}}{\alpha,\beta,\gamma}, \quad \forall \ti{\varphi} \in \tS(\Gamma).
$$
\item The MFT applied to a   tempered distribution $T \in \cSd'(Z)$ is defined as
$$
<\cF(T),\ti{\varphi}>_{\tS',\tS} = \cF(T)(\ti{\varphi}) := T(\cF^T(\ti{\varphi})) = <T, \cF^T(\ti{\varphi})>_{\cSd'\times\cSd},\quad \forall \ti{\varphi} \in \tS(\Gamma),
$$
where
$$
\cF^T (\ti{\varphi}) := \sum_{k\in \Z} \int_{\RR} \ti{\varphi}(k,l) e^{ - i (\ti{k} y_1+ \ti{l} y_2)} dl, \quad \forall \ti{\varphi} \in \tS(\Gamma),{\nobreakspace}\quad \forall \y \in Z. 
$$
In  the same way, we define the reciprocal operator denoted $\cF^{-1}$ which associates to each distribution $\ti{T}\in \tS'(\Gamma)$ a  distribution $\cF^{-1}(\ti{T})$ s.t.
$$
< \cF^{-1}(\ti{T}), \varphi >_{\cS'\times\cS} := < \ti{T} , \breve{\cF}(\varphi) >_{\tS'\times\tS}
$$
where $\breve{\ti{\varphi}}(k,l)=\ti{\varphi}(-k,-l)$.
\end{enumerate}
\end{defi}

We sum up properties extending classical results of Fourier analysis to our MFT.
\begin{proposition}\label{prop.fourier.dirac}
\begin{enumerate}[(i)]
\item If $T \in \cSd'(Z)$ then $\cF(T) \in \tS'(\Gamma)$. Similarly, if $\ti{T} \in \tS'(\Gamma)$ then $\cF^{-1}(\ti{T}) \in \cSd'(Z)$.
\item   The Dirac measure belongs to $\cSd'(Z)$ and ~:
$$
\cF( \delta_0 ) = 1,\quad \forall k \in \Z ,\quad \text{ a.e. }\,\ell \in \RR.
$$
\item the MFT acts on the derivatives in a polynomial fashion:
$$
\forall T \in \cSd'(Z) ,\quad \cF(\partial^{\alpha} T) = i^{|\alpha|} (\ti{k}^\alpha+\ti{l}^{\alpha}) \cF(T),
$$
for any multi-index $\alpha \in \NN^2$.
\item Plancherel's Theorem : if $f\in L^2(Z)$ then the Fourier transform $\cF(T_f)$ is defined by a function $\cF(f) \in l^2(\Z,L^2(\RR))$ i.e.
$$
\cF(T_f) = T_{\cF(f)} \quad \forall f \in L^2(Z) \text{ and } \nrm{\cF(f)}{l^2(\Z;L^2(\R))} = \nrm{f}{L^2(Z)}.
$$

\end{enumerate}
\end{proposition}
\subsection{The Green function for  periodic strips}

We aim at solving the fundamental equation:
\begin{equation}\label{eq.green.lap}
- \Delta G = \delta_0 \quad \text{ in } Z,
\end{equation}
stated here in the sense of tempered distributions. 
Then $G$ should satisfy, in $\tS(\Gamma)'$:
$$
(\ti{k}^2 + \ti{l}^2 ) \cF(G) = 1.
$$
\begin{proposition}\label{prop.green.func}
The Green function $G$ solving \eqref{eq.green.lap}
is a tempered distribution and it reads 
$$
G(\y)= \ud \left\{  
  \sum_{k \in \Z^*} \frac{ e^{-|\ti{k}| |y_2| + i \ti{k} y_1} }{|\ti{k}| } 
    - | y_2 | 
\right\} = G_1(\y) + G_2(\y) , \quad \forall \y \in Z,
$$
\noindent where
$$G_1(\y):=\ud\sum_{k \in \Z^*} \frac{ e^{-|\ti{k}| |y_2| + i \ti{k} y_1} }{|\ti{k}| }\,\,\,\mbox{ and }\,\,\,G_2(\y):=-\ud|y_2|.$$ 
Moreover, one has $G_1 \in L^1(Z)\cap L^2(Z)$ and thus $G\in L^2_{\loc}(Z)$. 
The Green function can be written in a more compact expression :
$$
G(\bfy)=  - \frac{1}{4 \pi } \ln ( 2 \left( \cosh ( 2 \pi y_2 ) - \cos( 2 \pi y_1 )\right) ).
$$
Notice that the Green function is odd with respect to $\y$.
\end{proposition}

\begin{proof}
 We define:
$$
%\left\{ 
%\begin{aligned}
G %& 
= \lim_{N\to \infty} G_N := \lim_{N\to \infty} \sum_{|k|<N} J_k e^{i\ti{k}y_1}, \text{ with } %\\
J_k  %& 
:= \int_{\R} \frac{e^{i\ti{l}y_2}}{\ti{k}^2 + \ti{l}^2} dl = \lim_{R \to \infty} \int_{-R}^R \frac{e^{i \ti{l}|y_2|}}{\ti{k}^2 + \ti{l}^2} dl .
%\end{aligned}
%\right.
$$
The absolute value in the last right hand side is added as follows:
$$
\int_{\RR}  \frac{e^{i\ti{l}y_2}}{\ti{k}^2 + \ti{l}^2} dl = \int_{0}^\infty  \frac{e^{i\ti{l}y_2}+e^{-i\ti{l}y_2}}{\ti{k}^2 + \ti{l}^2} dl = 2 \int_0^\infty   \frac{\cos({\ti{l}y_2})}{\ti{k}^2 + \ti{l}^2} dl = 2 \int_0^\infty   \frac{\cos({\ti{l}|y_2|})}{\ti{k}^2 + \ti{l}^2} dl.
$$
Proceeding as in example 1 p. 58 \cite{Book.ReSi}, one has, extending the integral to the complex plane, that
$$
\begin{aligned}
J_k &\equiv \lim_{R \to \infty} \oint_{C_R} \frac{e^{i \ti{z}|y_2|}}{\ti{k}^2 + \ti{z}^2} dz =  \lim_{R \to \infty} \oint_{C_R} \frac{e^{i \ti{z}|y_2|}}{(\ti{z}+i\ti{|k|} )(\ti{z}-i\ti{|k|} )} dz \\
& = \lim_{R \to \infty} \frac{1}{2\pi} \oint_{C_R} \frac{{\mathfrak f}(z)}{z-i|k|} dz = i {\mathfrak f}(i|k| ) = \frac{e^{-|\ti{k}| |y_2|}}{2 |\ti{k} |},
\end{aligned}
$$
where $C_R:= \{ z \in \CC; \; |z| = R \text{ and } \Im(z)> 0\} \cup ([-R;R]\times\{{\nobreakspace}0 \} )$ as depicted in fig. \ref{fig.demi.disque} and 
${\mathfrak f}(z):=e^{i \ti{z}|y_2|} / (\ti{z}+i\ti{|k|} )$, being a holomorphic function inside $C_R$. 
\begin{figure}[ht!] 
\begin{center}
\begin{picture}(0,0)%
\includegraphics{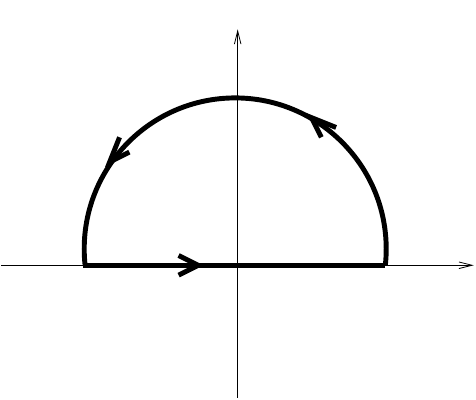}%
\end{picture}%
\setlength{\unitlength}{2072sp}%
\begingroup\makeatletter\ifx\SetFigFont\undefined%
\gdef\SetFigFont#1#2#3#4#5{%
  \reset@font\fontsize{#1}{#2pt}%
  \fontfamily{#3}\fontseries{#4}\fontshape{#5}%
  \selectfont}%
\fi\endgroup%
\begin{picture}(4344,3636)(529,-2188)
\put(3691,-1186){\makebox(0,0)[lb]{\smash{{\SetFigFont{6}{7.2}{\rmdefault}{\mddefault}{\updefault}{\color[rgb]{0,0,0}$(R,0)$}%
}}}}
\put(4546,-1231){\makebox(0,0)[lb]{\smash{{\SetFigFont{6}{7.2}{\rmdefault}{\mddefault}{\updefault}{\color[rgb]{0,0,0}$\Re(z)$}%
}}}}
\put(2296,1289){\makebox(0,0)[lb]{\smash{{\SetFigFont{6}{7.2}{\rmdefault}{\mddefault}{\updefault}{\color[rgb]{0,0,0}$\Im(z)$}%
}}}}
\put(3871,299){\makebox(0,0)[lb]{\smash{{\SetFigFont{6}{7.2}{\rmdefault}{\mddefault}{\updefault}{\color[rgb]{0,0,0}$C_R$}%
}}}}
\put(1216,-1186){\makebox(0,0)[lb]{\smash{{\SetFigFont{6}{7.2}{\rmdefault}{\mddefault}{\updefault}{\color[rgb]{0,0,0}$(-R,0)$}%
}}}}
\end{picture}%

\end{center}\caption{Path of integration in the complex plane}\label{fig.demi.disque}
\end{figure}
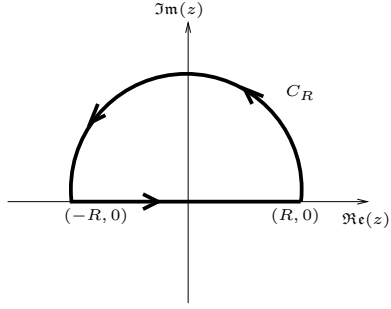
The 1D-Fourier transform of the tempered distribution $|y|$ is $-2 / \ti{l}^2$, thus one concludes formally that
$$
J_0 =\int_{\R} \frac{e^{i\ti{l}y_2}}{\ti{l}^2} dl = -\ud |y_2|.
$$
The $L^2$ bound is achieved thanks to the Parseval formula
$$
\nrm{G_1}{L^2(Z)}^2 = \nrm{\cF(G_1)}{\ell^2(\Z;L^2(\RR))}^2 =  \sum_{k\in\Z^*} \int_\R \frac{1}{(\ti{k}^2 + \ti{\ell}^2)^2} d\ell  \leq  C\sum_{k \in \Z^*} \frac{1}{|\ti{k}|^3} \leq C'.
$$
One recalls the expansion in series of the logarithm :
$$
\ln(1-z)=  - \sum_{k=1}^\infty \frac{z^k}{k},\quad \forall z \in \CC : |z| < 1.
$$
Thus for $\bfy \neq 0$
$$
\begin{aligned}
G_1(\y) & = - \frac{1}{2 \pi}  \Re \ln \left\{ 1 - e^{ 2 \pi ( -|y_2| + i y_1 )}\right\}= - \frac{1}{2 \pi}   \ln  \left| 1 - e^{ 2 \pi ( -|y_2| + i y_1 )}\right|\\
& = - \frac{1}{2 \pi}  \ln  \left\{ e^{-\pi |y_2|} \sqrt{2} \left( \cosh ( 2 \pi |y_2| ) - \cos( 2 \pi y_1 )\right)^\ud \right\} \\
& = \frac{|y_2| }{2} - \frac{1}{4 \pi } \ln ( 2 \left( \cosh ( 2 \pi y_2 ) - \cos( 2 \pi y_1 )\right) ),
\end{aligned}
$$
which gives the desired result for $G$.
\end{proof}

\noindent One can easily obtain the asymptotic behavior of the Green function. As $|y_2|$ tends to infinity, we have
\begin{equation}
\label{comportement_infini_G}
|G(\y)|\leq C|y_2|,\hspace{2mm}|\nabla G(\y)|\leq C\hspace{1mm}\mbox{ and }\hspace{1mm}|\partial^2_{\y^2} G(\y)|\leq C e^{-2\pi|y_2|}
\leq C\rho(y_2)^{-\sigma},\,\,\,\forall\sigma\in\R.
\end{equation} 
\noindent Moreover, as $|y_2|$ tends to infinity, we also have 
\begin{equation}
\label{comportement_infini_G1}
|G_1(\y)|\leq C\rho(y_2)^{-\sigma},\,\,\,\forall\sigma>0, \quad  \forall y_1 \in (0,1).
\end{equation}

\subsection{Convolution with the fundamental solution in weighted spaces}

\noindent Before proving weighted estimates on the Green function, we shall define 
the convolution if $f\in\cDd(Z)$, i.e. 
$$
G*f=\int_Z G(\x-\y)\,f(\y)\,d\y=\int_Z G(\y)\,f(\x-\y)\,d\y.
$$
Thus the convolution  $G*f$ belongs to $\cC_{\#}^\infty(Z)$. Moreover, thanks to \eqref{eq.green.lap},
we have $\Delta(G*f)=f$ in Z.

\subsection{Homogeneous estimates}

\begin{defi}\label{def.moy}
  For any function $f$ in $L^2(Z)$ the horizontal average $\ov{f}$ reads~:
$$
\ov{f}(y_2):= \int_0^1 f(y_1,y_2) dy_1.
$$
\end{defi}

\begin{lemma}\label{lem.zero.moy}
For any $h\in \cSp(Z)$ such that $\ov{h}\equiv0$  one has~:
$$
\partial^\alpha G_2 * h = 0, \text{ a.e. in }Z, \quad \alpha \in \{ 0,1,2 \}. 
$$
\end{lemma}

\begin{proof}
  $G_2$ is a tempered distribution, the convolution with $h\in \cSp(Z)$ makes 
sense~:
$$
\begin{aligned}
< G_2*h, \varphi >_{\cSp'(Z)\times\cSp(Z)}  %= < G_2 , \breve{(h*\varphi)} >_{\cSp'(Z)\times\cSp(Z)} \\
%& = \ud \int_{Z} |x_2 | \int_{Z} h(\bfx + \bfy) \varphi(\bfy) d \bfy d \bfx \\
&=  \ud \int_{\RR} |x_2 | \int_{Z} \ov{h}(x_2+y_2) dz \; \varphi(\bfy) d \bfy d x_2 =0.
\end{aligned}
$$
The same proof holds for derivatives as well. As this is true for every $\varphi \in \cSp(Z)$, the result is proved.
\end{proof}

\begin{proposition}\label{prop.nrm.h2.unif}
Let  $f \in L^2(Z)$. Then there exists a constant independent on $f$ s.t. 
$$
\nrm{G * (f-\ov{f})}{H^2_{\#}(Z)} \leq C\nrm{f}{H^{m-2}_{\#}(Z)},
$$
where $C$ is a constant independent on $f$.
\end{proposition}

\begin{proof}
First we show the lemma for $f\in \cSp(Z)$, then by density the result is extended to $L^2(Z)$ functions.
We set $h:=f-\ov{f}$.
Thanks to lemma above $\partial^{\bf \alpha} G * h = \partial^{\bf \alpha} G_1 * h $.
As $G_1$ belongs to $L^2(Z)$, one is allowed to apply the MFT in the strong sense.
It provides~: 
$$
\begin{aligned}
\nrm{\partial^{\bf \alpha}G * h}{L^2}^2 &= \nrm{\partial^{\bf \alpha}G_1 * h}{L^2}^2 =
\nrm{(i\ti{k})^{\alpha_1}(i \ell)^{\alpha_2} \hat{G_1} \hat{h}}{l^2L^2}^2 %\\
 =  \sum_{k \in \Z^*} \int_{\RR} \frac{|\ti{k}|^{2 \alpha_1}|\ti{\ell}|^{2\alpha_2}}{ (\ti{k}^2 + \ti{\ell}^2)^2 } \left| \hat{h}(k,\ell) \right|^2 d \ell  \\
& \leq C\sum_{k \in \Z^*} \int_{\RR} \left| \hat{h}(k,\ell) \right|^2 d \ell  
\equiv C \nrm{\hat{h}}{l^2(\Z;L^2(\RR))}^2 =  C \nrm{h}{L^2(Z)}^2
\end{aligned}
$$
where $\alpha$ is a multi index s.t. $|\alpha|\in\{0,1,2\}$. The summation in $k$ is performed in $\Z^*$, 
this is crucial when estimating $|\ti{k}|^{2 \alpha_1}|\ti{\ell}|^{2\alpha_2}/ (\ti{k}^2 + \ti{\ell}^2)^2$ by a constant.
%Let $h\in L^2(Z)$ and  $(h_n)_{n\in \N}$ a  sequence in $\cSp(Z)$  s.t. $h_n\to h$ in  $L^2(Z)$ strong, then $G * h_n$ is a Cauchy sequence in $H^2_{\#}(Z)$. Let $\omega:=\lim_{n\to \infty} G*h_n$, 
%thanks to Lemma \ref{lem.zero.moy}, for $|\alpha|\in \{0,1,2\}$ there holds~:
%$$
%<\partial^{\bf \alpha}G * h_n , \varphi >_{\cSp'(Z),\cSp(Z)}   = (\partial^{\bf \alpha} G_1 * h_n , \varphi )_{L^2(Z)} , \quad  \forall \varphi \in \cSp(Z).
%$$
%For $|\alpha|\in \{0,1\}$
%$$
%(\partial^{\bf \alpha}G_1 * h_n , \varphi )_{L^2(Z)} = (\partial^{\bf \alpha}G_1 , \breve{(\varphi * h_n)})_{L^2(Z)} \to  (\partial^{\bf \alpha} G_1 , \breve{(\varphi * h)})_{L^2(Z)}.
%$$
%This can be rewritten~:
%$$
% (\partial^{\bf \alpha} G_1 , \breve{(\varphi * h)})_{L^2(Z)} = (\partial^{\bf \alpha} G_1 * h , \varphi )_{L^2(Z)} = <\partial^{\bf \alpha} G_1 * h , \varphi >_{\cSp'(Z)\times\cSp(Z)}. 
%$$
%We  proved that $G * h_n$ converges strongly to $\omega$ in $H^2_{\#}(Z)$, $\partial^{\bf \alpha}\omega \equiv \partial^{\bf \alpha}G_1 * h$  a.e. in $Z$ for $|\alpha|\in \{0,1\}$. This ends the proof.
\end{proof}

\subsection{Non-homogeneous estimates}

We extend here the above results to the weighted context.
We start by stating three lemmas.

\begin{lemma}\label{lem.partial.convo}
We define the horizontal Fourier transform, given $f\in L^2(Z)$
$$
\cF_k(f)(k,y_2) := \int_0^1 f(y_1,y_2) e^{-i \ti{k} y_1} dy_1, \quad \forall k \in \Z, \quad \text{a.e } y_2 \in \RR.
$$
If we define the convolution operators~: for $g\in L^1(Z)\cap L^2(Z)$ and $f\in L^2(Z)$
$$
 g*_{\bfy} f := \int_{Z} g(\bfx -\bfy) f(\bfy) d\bfy ,\quad (g*_{y_2} f)(x_1,x_2,z_1) := \int_{\RR} g(x_1, x_2-y_2) f(z_1,y_2) dy_2,
$$
then one has 
$$
\cF_k (g *_{\bfy} f)(k,x_2) = (\cF_k g *_{y_2} \cF_k f)(k,x_2),\quad \forall k \in \Z, \quad \text{ a.e. } x_2 \in \RR.
$$
\end{lemma}

\begin{proof}
If $f \in L^2(Z)$ and $g \in (L^1\cap L^2)(Z)$ then the convolution with respect to both variables is well defined.
Setting
$$
(f*_{y_1} g)(x_1,x_2,z_2) := \int_{\RR} f(x_1-y_1, x_2) g(y_1,z_2) dy_1,
$$
almost everywhere in $y_2$ and for every $k \in \Z$ one has
$$
\cF_k\left( f*_{y_1} g (x_1,x_2-y_2,y_2)\right) = (\cF_k f) (k,x_2-y_2)  (\cF_k g) (k,y_2). 
$$
Integrals in the vertical direction commute with the horizontal Fourier transform
due to Fubini's theorem. So one shall write~:
$$
\begin{aligned}
\int_{\RR} \cF_k \left(f *_{y_1} g\right)(k,x_2-y_2,y_2) d y_2 & =  \cF_k \left(\int_{\RR} (f *_{y_1} g)(k,x_2-y_2,y_2) d y_2 \right) %\\
%& 
=  \cF_k (f *_{\bfy} g)(k,x_2),
\end{aligned}
$$
which ends the proof.
\end{proof}

\begin{lemma}\label{lem.partial.parseval}
  For $f \in L^2_\beta(Z)$, one has
$$
\nrm{f}{L^2_\beta(Z)}^2 = \sum_{k\in Z} \nrm{ \cF_k(f) }{L^2_\beta(\RR)}^2,\quad \forall \beta \in \RR.
$$
\end{lemma}

\begin{lemma}\label{coro.convo.expo}
Let $a \in \RR$ s.t. $|a|\geq 1$ then for any real $b$, one has
$$
I(x):= \int_{\RR} e^{-|a| |x-y|} \rho(y)^{-b} dy \leq \frac{C \rho(x)^{-b}}{|a|} ,\quad x \in \RR.
$$
\end{lemma}

\begin{proof}
We set  $|x|=R$ and let $R_0$ be a real such that $R_0>1$. We define three regions of the real line~:
$$
D_1 := B(0,R/2), \quad D_2 := B(0,2R)\setminus B(0, R/2), \quad D_3 := \RR \setminus B(0,2R)
$$
  Decomposing the convolution integral in these three parts, one gets~:
$$
I(x)=\sum_{i=1}^3  I_i(x),\quad \text{ where }  I_i(x)=\int_{D_i}e^{-|a| |x-y|}\rho(y)^{-b}\,dy.
$$
We have two cases to investigate~: $R>R_0$ and $R\leq R_0$.
\begin{enumerate}[(a)]
\item If $R>R_0$
\begin{enumerate}[(i)]
\item On $|y|\leq R/2$, $|x-y| \sim R$ so that
$$
I_1 (x) \sim e^{-|a|R}\int_{|y|\leq \frac{R}{2}}  (1+|y| )^{-b} dy. 
$$
According to the value of $b$, one has
$$
I_1 (x) \leq  C e^{-|a|R} \; 
\left\{ 
\begin{aligned} 
& 1& \text{ if } &b>1, \\
& R & \text{ if } &0\leq b\leq1, \\
& R^{1-b} & \text{ if } &b<0. 
\end{aligned}  
\right. 
$$
As $R>1$, there holds
$$
e^{-|a|R} \leq \frac{([\sigma]+1)!}{(|a|R)^\sigma}.
$$
Setting %respectively 
$$
\sigma(b) := b \, \chiu{]1,+\infty[}(b) + (1+b)\, \chiu{[0,1]}(b) + \chiu{]-\infty,0[}(b),
$$
one recovers the claim in that case. %that there exists $R_0$ s.t. 
\item On $R/2 \leq |y|\leq 2 R$, $|y| \sim R$ 
$$
I_2 (x) \sim R^{-b} \int_{\frac{R}{2} \leq |y| \leq 2 R} e^{-  |a| |x-y|} \, dy ,
$$
because $|x-y| \leq 3 R$, one has $I_2(x) \sim R^{-b} / |a|$.
\item On the rest of the line $|y|> 2R$, $|x-y| \sim |y|$, so that 
$$
I_3 (x) \sim \int_{|y|>2 R} e^{-|a | |y|} |y|^{-b} dy. %\leq   R^{-b} \int_{|y|> R}  e^{-|a | |y|} dy 
$$
Two situations occur~:
\begin{itemize}
\item either $b\leq 0$ then we set $\ti{b}:=[-b]+1$ and one has
$$
\begin{aligned}
\int_{|y|>2 R} e^{-|a | |y|} |y|^{-b} dy & \leq  C \frac{e^{ -|a| R }}{|a|} \left\{ R^{\ti{b}} \sum_{p=0}^{\ti{b}} \frac{\ti{b}!}{(\ti{b}-p)!} \left( |a| R \right)^{-p} \right\}  \\
& \leq  C \frac{e^{ -|a| R }}{|a|} \, R^{ \ti{b} } \, \ti{b}! \, \ti{b}  \leq  \frac{C(b)R^{ -b }}{|a|} .
\end{aligned}
$$
In the latter inequality we used  decreasing properties of the exponential function in order to fix the correct behavior for large $R$.
\item either $b>0$ and one has directly that
$$
 \int_{|y|>2 R} e^{-|a | |y|} |y|^{-b} dy \leq R^{-b} \int_{|y|> R}  e^{-|a | |y|} dy .
$$
\end{itemize}
\end{enumerate}
\item Otherwise, if $R<R_0$ then $\rho(x-y)\sim \rho(y)$, and thus 
$$
I(x) \sim \int_{\RR} e^{-|a | |y | } \rho(y)^{-b} dy < \frac{C}{|a|}.
$$
\end{enumerate}
\end{proof}
\begin{proposition}\label{prop.weight.h1}
Let us set $\beta$ a positive real number. Then for every $f \in L^2_\loc(Z)$ s.t. $f-\ov{f}\in L^2_\beta(Z)$
$$
\nrm{\partial^\alpha G*(f-\ov{f})}{L^2_{\beta'}(Z)}\leq C \nrm{f-\ov{f}}{L^2_\beta(Z)},\quad \forall \beta' < \beta -\ud, \quad |\alpha| \in \{0,1\},
$$
where the constant $C$ is independent of the data.
\end{proposition}

\begin{proof}
Set again $h=f-\ov{f}$ and suppose that $h\in\cSp(Z)$.
By Lemma \ref{lem.partial.parseval}~:
$$
\begin{aligned}
\nrm{\partial^\alpha G*h}{L^2_{\beta'}(Z)}^2 &= \nrm{\partial^\alpha G_1*h}{L^2_{\beta'}(Z)}^2 =\sum_{k\in \Z^*} \nrm{ \cF_k(\partial^\alpha_{\bfy}G_1*_{\bfy}h) }{L^2_{\beta'}(\RR)}^2 \\
& =  \sum_{k\in \Z^*} \nrm{ (i\tik)^{\alpha_1}\cF_k (\partial_{y_2}^{\alpha_2}G_1)*_{y_2}\cF_k(h) }{L^2_{\beta'}(\RR)}^2 .
\end{aligned}
$$ 
For all $k\in \Z^*$, thanks to Lemma \ref{coro.convo.expo}, a.e. $\bfx$ in  $Z$, 
$$
\begin{aligned}
\left| (i\tik)^{\alpha_1}\cF_k(\partial^{\alpha_2}_{y_2} G_1)*_{y_2}\cF_k(h) \right|& \leq 
\int_{\RR} \frac{e^{-|\tik| |y_2-x_2|}}{|\tik|^{1-|\alpha|}} \cF_k(h)(k,y_2) dy_2  \\
& \hspace{-2cm} \leq  C \nrm{\cF_k(h)(k,\cdot)}{L^2_\beta(\RR)} \frac{\rho(x_2)^{-\beta}}{|\tik|^{\td-|\alpha|}}%\\
%& 
\leq   C \nrm{\cF_k(h)(k,\cdot)}{L^2_\beta(\RR)} \rho(x_2)^{-\beta},
\end{aligned}
$$
for $\alpha$ multi-index s.t. $|\alpha| \in \{0,1\}$.
Then
$$
\nrm{ \partial^\alpha G * h }{L^2_{\beta'}(Z)}^2 \leq \sum_{k \in \Z^*} \nrm{\cF_k(h)(k,\cdot)}{L^2_\beta(\RR)}^2 \int_{\RR} \rho(x_2)^{2\beta'-2\beta} dx_2 \leq C \nrm{f}{L^2_\beta(Z)} ,
$$
if $\beta'<\beta- 1/2$. As in Proposition \ref{prop.nrm.h2.unif}, the result above shall be extended to $L^2(Z)$ functions by density arguments.
\end{proof}
\begin{thm}\label{thm.ortho.poly}
Let $f\in L^2_{\alpha}(Z) \bot \PP_1'$, for any $\alpha > \frac{3}{2}$, one has
$$
\nrm{G_2 * f }{L^2_{\alpha-2-\e}(Z)} \leq C \nrm{f}{L^2_\alpha(Z)}.
$$
\end{thm}

\begin{proof}
Since the proof is essentially 1D, we consider functions defined in $\RR$.
The extension to $Z$ is straightforward. We proceed by duality, namely we expect that
$$
\sup_{ \varphi \in L^2_{2-\alpha+\e}(Z)}  \frac{( G * f, \varphi)}{\nrm{\varphi}{L^2_{2-\alpha+\e}(Z)}} < \infty.
$$
We choose $f \in L^2_{\alpha}(\RR)\bot\PP_1'$ and $(\varphi_\delta)_\delta \in \cSp(\RR)$ s.t. $\varphi_\delta \to\varphi$  in $L^2_{2-\alpha +\e}(\RR)$,
then $|x|*\varphi_\delta$ is infinitely differentiable and one can apply the Taylor expansion 
with an integral remainder~:
$$
\begin{aligned}
&\int_{\RR} (|x|* \varphi_\delta)(x) f(x)dx  \\
&=  \int_{\RR} \left\{ (|x|*\varphi_\delta)(0)+ (|x|*\varphi_\delta)'(0).x + \int_0^x (|x|*\varphi_\delta)''(s)(x-s) ds \right\} f(x) dx \\
&= \int_{\RR} f(x) \int_0^{x} (|x|*\varphi_\delta)''(s)(x-s) ds dx  =   2 \int_{\RR} f(x) \int_0^{x} \varphi_\delta(s)(x-s) ds dx \\
& \leq C \nrm{ \varphi_\delta }{L^2_{2-\alpha+\e}(\RR)} \int_{\RR}\rho^{\alpha-\ud-\e}(x) |f(x) |dx 
\leq C  \nrm{ \varphi_\delta}{L^2_{2-\alpha+\e}(\RR)} \nrm{f}{L^2_{\alpha}(\RR)} .
\end{aligned}
$$
In the second line we used the orthogonality condition on $f$ 
in order to cancel the first two terms in the Taylor expansion.
An easy computation shows that if  $\varphi_\delta \to\varphi$ 
$$
\ell:= \lim_{\delta \to0} \int_{\RR} (|x|* \varphi_\delta)(x) f(x)dx =  2 \int_{\RR} f(x) \int_0^{x} \varphi(s)(x-s) ds dx,\quad  \forall f\in L^2_{\alpha}(\RR) \bot \PP_1'.
$$
By Fubini theorem, one is then allowed to write~:
$$
\begin{aligned}
& 2 \int_\RR f(x) \int_0^{x} \varphi(s)(x-s) ds dx  \\
%& = 
%2 \int_{\RR_{x}^+} f(x) \int_0^{x} \varphi(s) (x -s ) ds d x  + \int_{\RR_{x}^-} f(x) \int_0^{x} \varphi(s) (x -s ) ds d x  \\ 
&  = 2 \left\{ \int_{\RR_{s}^+} \varphi (s)  \int_s^\infty f(x) (x -s ) d x  ds 
- \int_{\RR_{s}^-} \varphi (s) \int_{-\infty}^s f(x) (x -s ) d x ds \right\}  \\
&= -2 \int_{\RR_{s}}\varphi (s) \int_{-\infty}^s f(x) (x -s ) d x ds  \\
& =  \int_{\RR_{s}}\varphi (s) \left\{\int_s^\infty  f(x) (x -s ) d x - \int_{-\infty}^s f(x) (x -s ) d x \right\}ds \\
& =  \int_{\RR} \varphi(s) \int_{\RR_x} f(x)|x-s| dx\,ds \equiv  \int_{\RR} (|x|*f) (s) \varphi(s) ds, 
\end{aligned}
$$
which ends the proof.
\end{proof}

\begin{lemma}\label{lem.partie.moyenne}
If $\alpha > \ud$, and if $f \in L^2_{\alpha}(Z)\bot \RR$, then
$$
\nrm{ \sgn * \ov{f}}{L^2_{\alpha-1-\epsilon}(Z)}\leq C \nrm{ f}{L^2_{\alpha}(Z)} .
$$
\end{lemma}

\begin{proof}
In a first step $f$ does not   satisfy the polar condition.
Under sufficient integrability conditions, one writes ~:
$$
(\sgn(x_2) * \ov{f})(\bfx)  = - \int_{y_2 < x_2} \ov{f} dy_2 + \int_{y_2 > x_2} \ov{f} dy_2.
$$
For $x_2$ tending to infinity, $\ov{f}*\sgn(x_2)$ behaves as
$\sgn(x_2) \int_\RR \ov{f}$. Indeed
$$
\begin{aligned}
&\left| (\sgn(x_2) * \ov{f})(\bfx) -\sgn(x_2) \int \ov{f} \right| = 2
\left|  \left( \int^{x_2}_{-\infty} \ov{f}(s) ds \right) \chiu{\RR_-}(x_2) - \left( \int_{x_2}^\infty \ov{f}(s) ds \right) \chiu{\rr}(x_2) \right| \\
& \hspace{1cm}\leq 2\nrm{f}{L^2_\alpha(Z)}
\left( 
         \left\{\int_{-\infty}^{x_2} \rho^{-2\alpha}(y_2) dy_2 \right\}^{\ud}  \chiu{\RR_-}(x_2)+  
         \left\{\int^{\infty}_{x_2} \rho^{-2\alpha}(y_2) dy_2 \right\}^{\ud} \chiu{\RR_+}(x_2) 
\right) \\
& \hspace{1cm}\leq C \nrm{f}{L^2_\alpha(Z)} \rho^{\ud-\alpha}(x_2) .
\end{aligned}
$$
Then taking the square, multiplying by $\rho^{2\beta}$ and integrating wrt $x$
$$
\nrm{  \sgn(x_2) * \ov{f}-\sgn (x) \int \ov{f} }{L^2_\beta(Z)}^2  \leq C  \nrm{f}{L^2_\alpha(Z)}^2 \int_{\RR} \rho^{1-2\alpha+2\beta}(x_2) dx_2,
$$
which is bounded, provided that $\beta < \alpha -1$. Taking into account the polar condition $f\bot \R$ gives then the claim. 
\end{proof}

\noindent Similarly to Theorem \ref{thm.ortho.poly}, thanks to the previous lemma, one gets that
\begin{thm}\label{thm.ortho.cst}
Let $f\in L^2_{\alpha}(Z) \bot \RR$ then for any $\alpha \in ]\ud,\td]$ one has
$$
\nrm{G_2 * f}{L^2_{\alpha-2-\e}(Z)/\R} \leq C \nrm{f}{L^2_\alpha(Z)}.
$$
\end{thm}

\begin{proof}
As previously, the proof is essentially 1D. Let us choose $f \in L^2_{\alpha}(\RR)\bot\RR$ and $(\varphi_\delta)_\delta \in \cSp(\RR)\bot\RR$ s.t. $\varphi_\delta \to\varphi$  in $L^2_{2-\alpha +\e}(\RR)\bot\RR$
then $|x|*\varphi_\delta$ is infinitely differentiable and one can apply the Taylor expansion 
with the integral rest ~:
$$
\begin{aligned}
&\int_{\RR} (|x|* \varphi_\delta)(x) f(x)dx  \\
&=  \int_{\RR} \left\{ (|x|*\varphi_\delta)(0)+ \int_0^x (\varphi_\delta*\sgn)(s) ds \right\} f(x) dx \equiv  \int_{\RR} f(x) \int_0^x (\varphi_\delta*\sgn)(s) ds dx.
\end{aligned}
$$
Proceeding as in the proof of  Lemma \ref{lem.partie.moyenne}, one has
$$
\left| \int_0^x \varphi_\delta * \sgn (s) \, ds\right|\,\leq \nrm{\varphi_\delta}{L^2_{2-\alpha+\e}(Z)} \rho^{\alpha -\ud - \e}(x),
$$
leading to 
$$
\nrm{ \int_0^x (\varphi_\delta*\sgn)(s) ds }{L^2_{-\alpha}(\RR)} \leq C \nrm{\varphi_\delta}{L^2_{2-\alpha+\e}(\RR)},
$$
with $\alpha < \td + \e$. Thus one has
$$
\left| \int_{\RR} (|x|* \varphi_\delta)(x) f(x)dx  dy \right| \leq C \nrm{f}{L^2_\alpha(\RR)}\nrm{\varphi_\delta}{L^2_{2-\alpha+\e}(\RR)}\,.
$$
Moreover, one needs that $\RR \in L^2_{-\alpha}(\RR)$, which is true if $\alpha>\ud$. 
It is not difficult to prove by similar arguments that
$$
\lim_{\delta\to 0} \int_0^x (\varphi_\delta *\sgn)(s) ds   =  \int_0^x (\varphi *\sgn)(s) ds  
$$
strongly in the $L^2_{-\alpha}(\RR)$ topology. 
As in the proof of Lemma \ref{lem.partie.moyenne}, by Fubini, 
$$
\int_0^x (\sgn * \varphi)(s) ds = 2 \int_0^x \left\{ \int_{-\infty}^s \varphi(t)dt \chiu{s<0}(s) -  \int^{\infty}_s \varphi(t)dt \chiu{s>0}(s) \right\}ds =: 2  \int_0^x {\mathfrak g} (s) ds. 
$$
As ${\mathfrak g}$ is a $C^1$ function on any compact set in $\RR$, 
one can integrate by parts on $(0,x)$~: 
$$
 \int_0^x {\mathfrak g} (s) ds = \left[ {\mathfrak g}(s) s\right]_0^x -  \int_0^x {\mathfrak g}' (s) \,s \,ds .
$$
Using this in the right hand side of the  previous limit, one writes
$$
\begin{aligned}
J:= &  \int_\RR f(x) \int_0^{x} (\varphi *\sgn)(s) ds dx = 2   \int_\RR f(x) \left\{  {\mathfrak g}(x)x - \int_0^x {\mathfrak g}' (s) \,s \,ds \right\} dx\\
& = 2 \left\{ \int_{\RR_-}  f(x)\, x \int_{-\infty}^x \varphi(t) dt  dx -  \int_{\rr}  f(x) x \int^{\infty}_x \varphi(t) dt  dx - \int_{\RR}  f(x) \int_0^x \varphi(s) \,s \,ds  \right\} \\
& = 2 \int_{\RR}  f(x)\, x \int_{-\infty}^x \varphi(t) dt  dx - 2 \int_{\RR}  f(x) \int_0^x \varphi(s) \,s \,ds \,dx =: A - B.
\end{aligned}
$$
By H{\"o}lder estimates the integrals above are well defined. Fubini's theorem allows then to switch integration order.
Using the orthogonality condition on $\varphi$ and on $f$, one may easily show that
$$
\left\{ 
\begin{aligned} 
A &= 2 \int_{\RR_x}  f(x)\, x \left(\int_{-\infty}^x \varphi(t) dt\right)  dx = \int_\RR f(x) x \left\{ \int_{- \infty}^x \varphi(s) ds - \int_x^\infty \varphi(s) ds \right\} dx \\
&= \int_{\RR_s} \varphi(s) \left\{ \int_{s}^\infty xf(x) dx - \int_{-\infty}^s x f(x)dx \right\} ds, \\
B & = 2 \left\{ \int_{\rr} f(x) \int_0^x s\, \varphi(s) \, ds \, dx - \int_{\RR_-} f(x) \int_x^0 s\, \varphi(s) \, ds \, dx \right\}  \\
& = 2 \left\{ \int_{\rr} s\, \varphi(s) \int_s^\infty f(x) dx ds - \int_{\RR_-}  s\, \varphi(s)\int_{-\infty}^s f(x) dx ds \right\} \\
& = -2 \int_{\RR}  s\, \varphi(s) \int_{-\infty}^s f(x) dx ds = \int_{\RR}  s\, \varphi(s) \left\{  \int_s^\infty f(x) dx ds - \int_{-\infty}^s f(x) dx \right\} ds.
\end{aligned}
\right.
$$
These computations give~:
$$
\begin{aligned}  A - B& = \int_{\RR_s} \varphi(s) \left\{ \int_{s}^\infty (x-s) f(x) dx - \int_{-\infty}^s (x-s) f(x)dx \right\} ds \\
&= \int_\RR \varphi(s) \int_\RR |x-s| f(x) dx ds = \int_{\RR} \varphi(s) (|x|*f)(s) ds.
\end{aligned}
$$
And because
$$
\inf_{\lambda \in \RR} \nrm{|x|*f+\lambda}{L^2_{\alpha -(2 +\e)}(\RR)} = \sup_{\varphi \in L^2_{(2+\e) -\alpha}(\RR)\bot\RR} \frac{ (|x|*f ,\varphi)}{\nrm{\varphi}{L^2_{(2+\e) -\alpha}(\RR)}},
$$
the final claim follows.
\end{proof}

\begin{thm}\label{thm.convo.reg} Assume $\alpha>1/2$ and recall that $q(0,-\alpha)$ is defined by \eqref{def_q_bis}. Then the operator defined by the convolution with the fundamental solution $G$ is a mapping from $L_{\alpha}^2(Z)\bot\PP_{q(0,-\alpha)}^{'\Delta}$ on $\wsd{1}{\alpha-1-\e}{Z}\cap\wsd{2}{\alpha - 2-\e}{Z}/ \PP_{[3/2-\alpha]}^{'\Delta}$ for any $\e>0$.
\end{thm}

\begin{proof}
One follows  the same lines as in the proof of Theorem 1 p. 786 in \cite{McOwen.1979}. 
More precisely, by Theorems \ref{thm.ortho.poly} and \ref{thm.ortho.cst},  Prop. \ref{prop.weight.h1} and Lemma \ref{lem.partie.moyenne}, the convolution with $G$ maps also
  $L_{\alpha}^2(Z)\bot\PP_{q(0,-\alpha)}^{'\Delta}$ on $\wsd{1}{\alpha-1-\e}{Z}\slash \PP_{[3/2-\alpha]}^{'\Delta}$.
Thus for $f\in L_{\alpha}^2(Z)\bot\PP_{q(0,-\alpha)}^{'\Delta}$, let $u=G* f \in L^2_{\alpha-2-\e,\#}(Z)/\PP_{[3/2-\alpha]}^{'\Delta}$.
Then for any $\phi \in \cDp(Z)$, $\<\Delta (G*f) , \phi\>=\< f , G*\Delta \phi\> = \<f,\phi\>$ which implies that $\Delta u = f$ in 
the sense of distribution. But since $u \in L^2_{\alpha-2-\e,\#}(Z)/\PP_{[3/2-\alpha]}^{'\Delta}$ and $\Delta u \in L_{\alpha}^2(Z)\bot\PP_{q(0,-\alpha)}^{'\Delta}$, 
using a dyadic partition of unity \cite{CoDauNi.11.Book} and standard inner regularity results (see for instance a similar 
proof in Theorem 3.1 in \cite{NiWa.73} or chap I in \cite{CoDauNi}), one has that:
$$
\begin{aligned}
\sum_{|{\bf \gamma}|=2} \nrm{ \rho^{-\gamma} D^{\bf \gamma} u }{L^2_{\alpha-\e}(Z)}^2 & \leq C \left( \nrm{ \Delta u }{L^2_{\alpha-\e}(Z)}^2 + \nrm{u}{\wsd{1}{\alpha-1-\e}{Z}}^2 \right),%\\
%&  \leq C \left( \nrm{f}{L^2_{\alpha}(Z)}^2 + \nrm{u}{\wsd{1}{\alpha-1-\e}{Z}}^2 \right),
\end{aligned}
$$
where the constant does not depend on $u$. These estimates are not isotropic with respect to the weight. 
Taking the lowest weight in front of the second order derivative ends the proof.
\end{proof}

\subsection{Convolution and duality}

We generalize the latter convolutions to weak data, namely, when 
$f \in \wsd{-1}{\alpha}{Z}$. 
\begin{proposition}\label{prop.conv.dual} For $\beta < \ud $ and $f \in \wsd{-1}{\beta}{Z}\bot\pP_{[\ud+\beta]}$ one has~:
$$
\nrm{ G * f}{\eld{\beta-1-\e}{Z}/\pP_{[\ud-\beta] }} \leq C \nrm{f}{\wsd{-1}{\beta}{Z}},
$$
where $C$ is independent on $f$.  
\end{proposition}

\begin{proof}
Using Theorem \ref{thm.convo.reg} with $\alpha=(1+\e)-\beta$, which is possible since 
$\beta<1/2$, one has
$$
\begin{aligned}
\< f,G*\varphi \>_{\wsd{-1}{\beta}{Z}\bot \pP_{[\ud+\beta]}  \times \wsd{1}{-\beta}{Z}\slash \pP_{[\ud+\beta]}} &  \leq 
\nrm{ f}{\wsd{-1}{\beta}{Z}} \nrm{G*\varphi}{\wsd{1}{-\beta}{Z}\slash \pP_{[\ud+\beta]}}\\
& \leq  C \nrm{ f}{\wsd{-1}{\beta}{Z}} \nrm{\varphi}{\eld{(1+\e)-\beta}{Z}\bot \pP_{q(0,\beta-(1+\e))}}.
\end{aligned}
$$
as $\e$ is positive and arbitrarily small $q(0,\beta-(1+\e))=[1/2-\beta]$.
Hence taking the supremum over all $\varphi\in \eld{\alpha}{Z}\bot \pP_{[1/2-\beta]}$, ends the proof.
\end{proof}

\section{The Laplace equation in a periodic infinite strip}
\label{section_Laplace}

In this section we study the problem 
\begin{equation}\label{Laplace}
\left\{ 
\begin{aligned}
-&\Delta u = f \text{ in } Z,  \\
&u\text{ is 1-periodic in the }y_1\mbox{ direction}
\end{aligned}
\right.
\end{equation}
in the variational context. Firstly, 
we characterize of the kernel of the Laplace operator.

\begin{proposition}
\label{prop_noyau}
Let $m\geq1$ be an integer, $\alpha$ be a real number and $j=\min\{q(m,\alpha),1\}$ where $q(m,\alpha)$ is defined by \eqref{def_q}. A function $u\in\Hmad(Z)$ satisfies 
$\Delta u=0$ if and only if $u\in\pP_j$.
\end{proposition}

\begin{proof}
Since $j\leq1$, it is clear that if $u\in\pP_j$, then $\Delta u=0$.
Conversely, let $u\in\cSd'(Z)$ satisfies $\Delta u=0$.  We apply the MFT~:
in the sense defined in Definition \ref{defi.fourier.distrib}, one has 
$$
(\ti{k}^2 + \ti{l}^2) \cF(u) =0 ,\quad \forall k \in \Z , \text{{\nobreakspace}a.e. } l \in \RR,
$$
which implies that 
$$
\cF(u)  = \begin{cases} \sum_{j =0}^{p}  \delta_{y_2=0}^j (l)& \text{ if } k \equiv 0 \\
0 & \text{otherwise}
\end{cases}
$$
where $p$ is a non negative integer.
A simple computation shows that $\cF^{-1} ( \cF(u) ) = \sum_{j=0}^p (i y_2)^j$, indeed
$$
\begin{aligned}
& \left< \cF^{-1}\left(\cF(u) \right) , \varphi \right> _{\cSd',\cSd} 
= \left< \cF(u), \breve{\cF}(\varphi)\right>_{\tS'(\Gamma),\tS(\Gamma)} 
 = \left< \sum_{j=0}^p \delta_0^j, \breve{\cF}(\varphi)(0,\cdot) \right>_{\cS'(\RR)\times\cS(\RR)} \\
& = \sum_{j=0}^p (-1)^j \left< \delta_0, \partial_j {\hat \varphi} \right>_{\cS'(\RR)\times\cS(\RR)} 
 = \sum_{j=0}^p (-1)^j \left< \delta_0, \widehat{((i y_2)^j \varphi)} \right>_{\cS'(\RR)\times\cS(\RR)} \\
& = \sum_{j=0}^p (-1)^j \int_{\RR} (iy_2)^j \varphi (\y) d \y 
 \quad = \left< \sum_{j=0}^p (-i y_2)^j , \varphi \right>_{\cSd'\times\cSd},
\end{aligned}
$$
where $\hat{\varphi}$ denotes the usual 1-dimensional Fourier transform on $\RR$.
\end{proof}

\noindent We  establish a Poincar{\'e}-Wirtinger's type inequality. 
\begin{lemma}\label{lem.poincare.writinger}
 For every $u \in \wsd{1}{\alpha}{Z}$ with $\alpha \in \RR$, one has
$$
\nrm{ u - \ov{u }}{L^2_\alpha (Z)} \leq \snrm{ u }{\wsd{1}{\alpha}{Z}},
$$
%where
%$$
%\ov{u}(y_2) := \int_0^1 u ( y_1, y_2) dy_1, \quad \text{{\nobreakspace}a.e. } y_2 \in \RR.
%$$
\end{lemma}
\begin{proof}
If $u \in  \wsd{1}{\alpha}{Z}$ then $\partial^\lambda u \in L^2_{\loc}(Z)$ for all $0 \leq |{\nobreakspace}\lambda | \leq 1$. In particular, 
almost everywhere in $y_2$ one has $\partial^\lambda u (\cdot, y_2)$ in $L^2(0,1)$. Applying Parseval in the $y_1$ 
direction one gets~:
$$
%\left\{ 
%\begin{aligned}
%& 
\nrm{ u - \ov{u} }{L^2(0,1)}^2  =  \sum_{k \in \Z^*} | \hat{u}(k) |^2 \text{{\nobreakspace}and }
%&
 \nrm{ \partial_{1} u}{L^2(0,1)}^2  =  \sum_{k \in \Z^*} |k|^2 | \hat{u}(k) |^2 %\\
%\end{aligned}
%\right.
$$
for almost every $y_2 \in \RR$. This obviously gives
$$
\nrm{ u - \ov{u} }{L^2(0,1)}^2  \leq \nrm{ \partial_{1} u}{L^2(0,1)}^2  \quad \text{{\nobreakspace}a.e. } y_2 \in \RR.
$$
Integrating then with respect to the vertical weight, one gets the desired estimates.
\end{proof}
\noindent In order to state  existence and uniqueness results for solutions of  \eqref{Laplace}, 
we  first deal with the Laplace operator in the truncated domain $Z_R:=]0,1[\times (]-\infty,-R[\cup]R,+\infty[)$~: 
\begin{equation}\label{trunc}
\left\{ 
\begin{aligned}
&-\Delta u = f & \text{ in } Z_R,  \\
& u = 0  & \text{ on } y_2= \pm R.  \\ 
\end{aligned}
\right.
\end{equation}
\begin{lemma}\label{lem.inf.sup.trunc}
Let $\alpha$ be any real number.
If $f \in\wsd{-1}{\alpha}{Z} $ there exists $R(\alpha)$ large enough such that  problem \eqref{trunc} has a unique solution $u \in \wsd{1}{\alpha}{Z_{R(\alpha)}}$.
\end{lemma}

\begin{proof}
First we notice that $\wsd{-1}{\alpha}{Z}  \subset \wsd{-1}{\alpha}{Z_R} $.
Then the proof relies on an {\em inf-sup} argument similar to the one used in \cite{Ba.Siam.76} (the main difference being the nature of weights : in \cite{Ba.Siam.76} the author derives similar estimates for exponential weights on a half strip). Indeed, we set
$$
v := \omega^2  ( u - \ov{u}) + \psi(u),{\nobreakspace}\quad \omega:=\rho^\alpha, \text{ and  }
\psi(u) :=\left\{
\begin{aligned}
& \int_R^{y_2}\omega^2\partial_{y_2}\ov{u} (s)\,ds & \text{ if }y_2>R\\
& 0 &  \text{ if }|y_2|\leq R\\
& - \int_{y_2}^{-R}\omega^2 \partial_{y_2}\ov{u}(s)\,ds & \text{otherwise}.
\end{aligned}\right.
$$
First of all, we check that $v \in \wsd{1}{-\alpha}{Z_R}$. 
$$
\nabla v = \nabla (\omega^2) (u - \ov{u}) + \omega^2 \nabla \ov{u}.
$$
An application of Lemma \ref{lem.poincare.writinger} proves  that
$$
\nrm{\nabla v}{L^2_{-\alpha}(Z_R)} \leq C \snrm{u}{\wsd{1}{\alpha}{Z_R}}.
$$
which guarantees, through Hardy estimates, that indeed $v$ belongs to $\wsd{1}{-\alpha}{Z_R}$.
Secondly, we write
$$
a(u,v) = \int_{Z_R} \nabla u \cdot \nabla v \; d\y = \int_{Z_R} |\nabla u |^2 \omega^2 \; d\y  + \int_{Z_R} \nabla u \cdot \nabla (\omega^2) \; ( u - \ov{u})  \; d\y .
$$
Using Cauchy-Schwarz, we estimate the formula above as~:
$$
\int_{Z_R} \nabla u \cdot \nabla (\omega^2)\; ( u - \ov{u})  \; d\y \leq  2 \left( \int_{Z_R} |\nabla u |^2 \omega^2 \; d\y  \right)^{\ud} \left( \int_{Z_R} |u -\ov{u} |^2 |\nabla \omega|^2  \; d\y  \right)^{\ud}. 
$$
Since $|\nabla \rho|<1$, one also has that 
$$
 \frac{|\nabla \omega|^2}{\omega^2 }  \leq \frac{\alpha^2 }{\rho^2(R)} , \quad \forall \y \in Z_R,
$$
which gives then
$$
\int_{Z_R} \nabla u \cdot \nabla (\omega^2) \;  ( u - \ov{u})  \; d\y \leq  \frac{|\alpha|}{\rho(R)}\int_{Z_R} |\nabla u |^2 \omega^2 \; d\y .
$$
This in turn implies  that
$$
a(u,v) \geq \left(1-2 \frac{|\alpha|}{\rho(R)}\right)  \snrm{ u }{\wsd{1}{\alpha}{Z_R}}^2 \geq C(R,\alpha)  \snrm{ u }{\wsd{1}{\alpha}{Z_R}}  \snrm{ v }{\wsd{1}{-\alpha}{Z_R}}.
$$
For every fixed $\alpha$ there  exists $R$ large enough s.t. $C(R,\alpha)>0$.
This proves that the operator $a(\cdot,v)$ is onto from $\wsd{1}{\alpha}{Z}$ to $\wsd{-1}{\alpha}{Z}$. 
In the same way, the adjoint operator is injective. Indeed (taking $\omega=\rho^{-\alpha}$ above), 
for all $u \in \wsd{1}{-\alpha}{Z_R}$ there exists a $v \in  \wsd{1}{\alpha}{Z_R}$ s.t.
$$
a(u,v) \geq \left(1 - 2 \frac{|\alpha|}{\rho(R)}\right)  \snrm{ u }{\wsd{1}{-\alpha}{Z_R}}^2 \geq C(R,\alpha)  \snrm{ u }{\wsd{1}{-\alpha}{Z_R}}  \snrm{ v }{\wsd{1}{\alpha}{Z_R}}.
$$
and one concludes using the classical Banach-Babu{\v s}ka-Ne{\v c}as result (see Theorem 2.6 p. 85 in \cite{ErGu.04.book}). 
%Thanks to the symmetry of the with respect to $|\alpha|< 1/2$  the same result holds for the transposed operator.
\end{proof}

\noindent We next solve the Laplace equation in the domain $Z_R\cup Z^R$.

\begin{lemma}\label{lem.sauts}
Provided that $f \in \wsd{-1}{\alpha}{Z}$, there exists a unique solution $u_0 \in \wsd{1}{\alpha}{Z}$ solving:
\begin{equation}\label{eq.sys.sauts} 
\left\{ 
\begin{aligned} 
& -\Delta u_0 = f & \text{ in } Z_R\cup Z^R, \\
& u_0 = 0 & \text{ on } \{y_2 = \pm R\}. 
\end{aligned}  
\right. 
\end{equation}
Moreover, if $f\bot\R$, and setting 
$$
\ov{h}_\pm := \left< \left[ \partial_{y_2} u_0 \right] , 1 \right>_{H_{\#}^{-\ud}(\{y_2= \pm R\}),H_{\#}^{\ud}(\{y_2= \pm R\})}, 
$$
where the brackets $[ \cdot ]$ denote the jump across interfaces $\{y_2 =\pm R\}$, i.e. 
$$
 \left[ \partial_{y_2} u_0 \right]= \lim_{y_2\to R^+} \partial_{y_2} u_0(y_1,y_2) - \lim_{y2\to R^-} \partial_{y_2} u_0(y_1,y_2),
$$  
then $h$ satisfies~:
\begin{equation}\label{eq.rel.h}
\ov{h}_+ + \ov{h}_- \equiv 0.
\end{equation}
\end{lemma}

\begin{proof}
By Lemma \ref{lem.inf.sup.trunc} there exists a unique $u_0$ in $\wsd{1}{\alpha}{Z}$ for every $\alpha\in \RR$. 
By truncation and approximation (following the same steps as in Lemma \ref{lem_convergence} and Proposition \ref{prop_densite_critique}), we set
$$
f_{\delta} = \left(f \Phi_\delta \right)* \alpha_\delta  ,\quad \ti{f}_\delta := f_\delta-\ov{f}_\delta ,
$$   
where $\Phi_\delta$ and $\alpha_\delta$ are  chosen as in Proposition \ref{prop_densite}. An easy check shows that
$$
\ti{f}_\delta \in \cDp(Z), \quad \ti{f}_\delta \bot \RR =0. 
$$
We then compute $u_{0,\delta}$ solving \eqref{eq.sys.sauts} with the data $\ti{f}_\delta$. One has
\begin{equation}\label{eq.esti.cvg.reg}
%\left\{
\begin{aligned}
 &\nrm{u_{0,\delta}}{\wsd{1}{\alpha}{Z}}  = \nrm{u_{0,\delta}}{\wsd{1}{\alpha}{Z_R\cup Z^R}} \leq C\nrm{\ti{f}_{\delta}}{\wsd{-1}{\alpha}{Z_R\cup Z^R}}   %\\
% & 
\leq C\nrm{\ti{f}_{\delta}}{\wsd{-1}{\alpha}{Z}}  \leq C \nrm{f}{\wsd{-1}{\alpha}{Z}},\\
& \nrm{u_{0,\delta}-u_0}{\wsd{1}{\alpha}{Z}}   \leq  C \nrm{\ti{f}_\delta -f}{\wsd{-1}{\alpha}{Z}} ,\\
& \nrm{\left[\partial_{y_2} u_{0,\delta} \right]-\left[\partial_{y_2} u_0 \right]}{H^{-\ud}(\{y_2=\pm R\})}  \leq C \nrm{u_{0,\delta}-u_0}{\wsd{1}{\alpha}{Z}}\leq  C \nrm{\ti{f}_\delta -f}{\wsd{-1}{\alpha}{Z}}.
\end{aligned}
%\right.
\end{equation}
Moreover, for every $\varphi \in \cDp(Z)$, one might write 
$$
\begin{aligned}
\< - \Delta u_{0,\delta} , \varphi \> &= - \< u_{0,\delta} ,\Delta \varphi \> = - \int_{Z_R \cup Z^R} u_{0,\delta} \Delta \varphi d\y \\
& = - \int_{(0,1)}  \left[ \partial_{y_2} u_{0,\delta} \right]_{\{y_2=\pm R\}} \varphi(y_1,\pm R)\, dy_1 + \<\ti{f}_\delta ,\varphi\> , 
\end{aligned}
$$
the latter equality being true because $u_{0,\delta}$ is in the domain of the operator~:  the Green formula holds. 
By density, one extends the above equality to test functions in $\wsd{1}{-\alpha}{Z}$, giving :
$$
 - \< u_{0,\delta} ,\Delta \varphi \> = - \< [\partial_{y_2}  u_{0,\delta} ] , \varphi \>_{\{y_2=\pm R\}} + \<\ti{f}_\delta ,\varphi\> , \quad \forall \varphi \in \wsd{1}{-\alpha}{Z}.
$$
As $\alpha$ is chosen s.t. $\RR \subset \wsd{1}{-\alpha}{Z}$, one takes $\varphi\equiv 1$ in the above formula. This leads to :
$$
\ov{h}_{+,\delta} + \ov{h}_{-,\delta} := \< [\partial_{y_2}  u_{0,\delta} ] , 1 \>_{\{y_2= R\}} + \< [\partial_{y_2}  u_{0,\delta} ] , 1 \>_{\{y_2= -R\}}  \equiv 0,
$$
where the brackets denote the $H_{\#}^{-\ud}$,$H_{\#}^\ud$ duality product. 
Thanks to \eqref{eq.esti.cvg.reg} and by continuity,  the result holds when passing to the limit with respect to $\delta$.
\end{proof}

%\noindent We continue by stating a lift up result of the jumps across the interfaces $\{y_2=\pm\R\}$.
When extending $u_0$ on Z, it solves 
\begin{equation}\label{eq.saut}
- \Delta u_0 = f + \delta_{\{y_2= \pm R\}} h_{\pm} ,\quad \text{ in } Z.
\end{equation}
In the next paragraph, we compute a lift that cancels the lineic Dirac mass in the right hand side of \eqref{eq.saut}.
\begin{lemma}\label{lem.sauts.seuls}
  Let $u_0$ be the unique solution of problem \eqref{eq.sys.sauts}. For $\alpha \leq 0$ there exists 
a solution $w$ in $\wsd{1}{\alpha}{Z}$ of 
\begin{equation}\label{eq.sauts.seuls} 
\left\{ 
\begin{aligned} 
& - \Delta w = 0  & \text{ in } Z_R\cup Z^R, \\
& \left[ \partial_{y_2} w \right] = -\left[ \partial_{y_2} u_0 \right] =: - h_{\pm} & \text{ on } \{y_2 = \pm R\}. 
\end{aligned}  
\right. 
\end{equation}
Moreover, we have 
$$\|w\|_{H_{\alpha,\#}^1(Z)}\leq C\sum_{\pm}\left(\|h_{\pm}\|_{H_{\#}^{-\ud}(\{y_2=\pm R\})}^2\right)^{1/2}.$$
\end{lemma}

\begin{proof}
Setting 
$$
\begin{aligned}
w := G * ( h_{\pm} \delta_{\{y_2=\pm R\}} )  & = 
G* ((h_\pm - \ov{h}_\pm)\delta_{\{y_2=\pm R\}})  + G* (\ov{h}_\pm\delta_{\{y_2=\pm R\}}) \\
& = G_1 * (h_\pm\delta_{\{y_2=\pm R\}} )+ G_2 * (\ov{h}_\pm\delta_{\{y_2=\pm R\}}),
\end{aligned}
$$
which is a tempered distribution ($G \in \cSp'(Z)$ and $h_\pm \delta_{\{y_2=\pm R\}} \in {\cal E}'_\#(Z)$), $w$ is explicit 
after some computations and reads:
$$
w = \sum_{\pm,k\in \Z^*} \frac{\cF_k ( h_\pm ) e^{i k y_1 -|k| |y_2\pm R|}}{|k|} - \left\{\frac{\ov{h}_+ }{ 2 } | y_2 - R |  +  \frac{\ov{h}_- }{ 2 } | y_2 + R | \right\}= I_1 + I_2,
$$
where $$I_1:=\sum_{\pm,k\in \Z^*} \frac{\cF_k ( h_\pm ) e^{i k y_1 -|k| |y_2\pm R|}}{|k|}\mbox{ and }I_2:= -\frac{\ov{h}_+ }{ 2 } | y_2 - R |  -  \frac{\ov{h}_- }{ 2 } | y_2 + R |.$$

One computes the $L^2_{\alpha -1} (Z)$ norm of $I_1$:
$$
\begin{aligned}
\int_{Z} I_1^2 \rho(y_2)^{2 \alpha- 2} d \y \leq \int_{Z} I_1^2 d\y & \leq 
C \sum_{\pm,k\in Z^*} \frac{| \cF_k ( h_\pm ) |^2 }{(1 + |k|^2)^\ud}  
\sum_{k\in Z^*} \int_{\RR} \frac{e^{-2 |k| |y_2\pm R|}}{|k|} d\y, \\
& \hspace{-1cm}\leq C' \sum_{\pm}\nrm{h_\pm}{H_{\#}^{-\ud}(\{y_2=\pm R\})}^2 \sum_{k \in \Z^*} \frac{1}{|k|^2}
\leq C'' \sum_{\pm}\nrm{h_\pm}{H_{\#}^{-\ud}(\{y_2=\pm R\})}^2.
\end{aligned}
$$
In the above relations, the first inequality holds since $\alpha\leq0$. Then using the Fourier transform one writes~:
$$
\cF\left(\nabla G_1 * (h_{\pm} \delta_{\{y_2=\pm R\}}) \right)  = 
\left\{ 
\begin{aligned} 
& i \frac{(k,l)^* }{(k^2 + l^2)} \cF_k\left(h_{\pm}\right), & \text{ if } k \in \Z^*, \quad   \text{ a.e. }l \in \RR , \\
& 0 & \text{ for } k=0.
\end{aligned}
\right.
$$
Next, taking the $l^2(\Z;L^2(Z))$ norm of the above expression, using again the fact that $\alpha\leq0$ and by Parseval, one gets:
$$
\begin{aligned}
\int_Z\rho(y_2)^{2\alpha}|\partial_{y_1}I_1|^2 d\y\leq\nrm{\partial_{y_1} G_1 * \left(h_{\pm} \delta_{\{y_2=\pm R\}} \right)}{L^2(Z)}^2 & =  \sum_{k\in \Z^*} \int_{R} \frac{k^2}{(k^2+l^2)^2}\; dl |\cF_k(h_\pm)|^2, \\
& \hspace{-1cm}= \sum_{k\in \Z^*} \frac{C}{|k|} |\cF_k(h_\pm)|^2 = \nrm{h_\pm}{H_{\#}^{-\ud}(\{y_2=\pm R\})}^2.
\end{aligned}
$$
A similar computation gives the  derivative with respect to $y_2$. It follows that
$$
\nrm{I_1}{\wsd{1}{\alpha}{Z}} \leq C\sum_{\pm}\left(\|h_{\pm}\|_{H_{\#}^{-\ud}(\{y_2=\pm R\})}^2\right)^{1/2}.
$$
Thanks to Lemma \ref{lem.sauts}, one can reduce $I_2$ to
$$
I_2 =  \mp  \ov{h}_{\pm} R \text{ when } |y_2| > R.
$$
Thus the claim follows since $I_1$ and $I_2$ both belong to $H_{\alpha,\#}^1(Z)$.
\end{proof}

\noindent We are now ready to state the first isomorphism result of the Laplace operator in $Z$.

\begin{thm}
\label{premiers_isomorphismes}  
Let $\alpha$ be a real number such that $-1/2\leq\alpha\leq1/2$.
Then the Laplace operator defined by 
\begin{equation}
\label{premiers_isomorphismes_1}
\Delta\,:\,H_{\alpha,\#}^1(Z)/\R\mapsto H_{\alpha,\#}^{-1}(Z)\bot\R,
\end{equation}
is an isomorphism.
\end{thm}

\begin{proof}
Let $f\in H^{-1}_{\alpha,\#}(Z)$. Observe that if $u\in H^1_{\alpha,\#}(Z)$ satisfies $\Delta u=f$, then due to the density of $\cDd(Z)$ in $H^1_{-\alpha,\#}(Z)$, for any 
$\varphi\in H^1_{-\alpha,\#}(Z)$, we have $$\<u,\Delta\varphi\>=\<f,\varphi\>.$$
Thus, we easily see that, under the assumptions on $\alpha$, for any $p\in\R\subset H^1_{-\alpha,\#}(Z)$,  the datum $f$ satisfies the necessarily compatibility condition 
$$\<f,p\>=0.$$
Now it is also clear that the Laplace operator defined by \eqref{premiers_isomorphismes_1} is linear and continuous. It is also injective since $\Delta u=0$ and $u\in\ H_{\alpha,\#}^1(Z)$ imply that $u$ is a constant (see Proposition \ref{prop_noyau}). It remains to prove that the operator is onto. Let $f$ be in $H_{\alpha,\#}^{-1}(Z)\bot\R$ and $R>0$ be a real number. Thanks to Lemmas \ref{lem.sauts} and \ref{lem.sauts.seuls}, one sets
$$
u := u_0 + w, 
$$
where $u_0$ solves \eqref{eq.sys.sauts} and $w$ lifts the jumps of $\partial_{y_2} u_0$ on $\{y_2=\pm R\}$ and thus solves \eqref{eq.sauts.seuls}. Results on  $u_0$ and  $w$ 
 end the proof. At this stage, the claim is proved for $\alpha \leq 0$.
By duality the results is true also for $\alpha>0$ which ends the proof.
 \end{proof}

\noindent It remains to extend the isomorphism result \eqref{premiers_isomorphismes_1} 
to any $\alpha\in\R$ and to any $m\in\Z$. For this sake, we state some regularity results for the Laplace operator in $Z$. 
\begin{thm}
\label{regularite_Hm}
Let $\alpha$ be a real number such that $-1/2<\alpha<1/2$ and let $\ell$ be an integer. 
Then the Laplace operator defined by 
\begin{equation}
\label{regularite_Hm_iso1}
\Delta\,:\,H_{\alpha+\ell,\#}^{1+\ell}(Z)/\R\mapsto H_{\alpha+\ell,\#}^{-1+\ell}(Z)\bot\R
\end{equation}
is onto.
\end{thm}
\begin{proof}
Owing to Theorem \ref{premiers_isomorphismes}, the statement of the theorem is true for $\ell=~0$.
Assume that it is true for $\ell=k$ and let us prove that it is still true for $\ell=k+1$.
The Laplace operator defined by \eqref{regularite_Hm_iso1} is clearly linear and continuous. It is 
also injective~: if $u\in H_{\alpha+k+1,\#}^{k+2}(Z)$ and $\Delta u=0$
then $u$ is constant.
To prove that it is onto, let $f$ be given in $H_{\alpha+k+1,\#}^k(Z)\bot\R$.
According to \eqref{inclusions_poids}, $f$ belongs to $H_{\alpha+k,\#}^{-1+k}(Z)\bot\R$.
Then the induction assumption implies that there exists $u\in H_{\alpha+k,\#}^{1+k}(Z)$ such that
$\Delta u=f$. Next, we have
\begin{equation}
\label{regularite_H2_eq1}
\Delta(\rho\partial_i u)=\rho\partial_i f+\partial_i u \Delta\rho+2\nabla\rho\nabla(\partial_i u).
\end{equation} 
Using \eqref{derivation_poids}, \eqref{inclusions_poids},  \eqref{derivation} and \eqref{multiplication_poids},
all the terms of the right-hand side belong to $H_{\alpha+k,\#}^{-1+k}(Z)$. This implies that 
$\Delta(\rho\partial_i u)$ belongs to $H_{\alpha+k,\#}^{-1+k}(Z)$. 
Therefore, $\Delta(\rho\partial_i u)$ also belongs to $H_{\alpha+k-1,\#}^{-2+k}(Z)$. 
Moreover, 
$u$ belonging to $H_{\alpha+k,\#}^{1+k}(Z)$ implies that $\rho\partial_i u$ belongs $H_{\alpha+k-1,\#}^k(Z)$ and for any 
$\varphi\in H_{-\alpha-k+1,\#}^{2-k}(Z)$, we have 
$$\langle\Delta(\rho\partial_i u),\varphi\rangle_{H_{\alpha+k-1,\#}^{-2+k}(Z)\times H_{-\alpha-k+1,\#}^{2-k}(Z)}
=\langle\rho\partial_i u,\Delta\varphi\rangle_{H_{\alpha+k-1,\#}^k(Z)\times H_{-\alpha-k+1,\#}^{-k}(Z)}.$$
\noindent Since $\R\subset H_{-\alpha-k+1,\#}^{2-k}(Z)$, we can take $\varphi\in\R$
which implies that $\Delta\varphi=0$. It follows that
$\Delta(\rho\partial_i u)$ belongs to $H_{\alpha+k,\#}^{-1+k}(Z)\bot\R$.
Thanks to the induction assumption, there exists $v$ in $H_{\alpha+k,\#}^{1+k}(Z)$ such that
$$\Delta v=\Delta(\rho\partial_i u).$$
\noindent Hence, the difference $v-\rho\partial_i u$ is a constant.
Since the constants are in $H_{\alpha+k,\#}^{1+k}(Z)$, we deduce that $\rho\partial_i u$ belongs 
to $H_{\alpha+k,\#}^{1+k}(Z)$ which implies that $u$ belongs to $H_{\alpha+k+1,\#}^{2+k}(Z)$. 

\end{proof} 

\begin{rmk}
{\rm 
Let us point out that the above theorem excludes the values $\alpha\in\{-\frac{1}{2},\frac{1}{2}\}$~: 
due to logarithmic weights, the space $H_{1/2,\#}^1(Z)$ is not included in $L_{-1/2}^2(Z)$.
By duality, this also implies that the space $L_{1/2}^2(Z)$ is not included in 
$H_{-1/2,\#}^{-1}(Z)$. Therefore, in \eqref{regularite_H2_eq1} for $\alpha=1/2$, the term $\partial_i u\Delta\rho$
does not belong to $H_{1/2,\#}^{-1}(Z)$. Thus, in this section, the extension of \eqref{premiers_isomorphismes_1} to any $\alpha\in\R$ will exclude some critical values of $\alpha$. We deal with these critical values in Section \ref{sec.critic}.  
} 
\end{rmk}

\begin{rmk}
{\rm
As an application of the above theorem for the particular case when $\ell=1$, we derive a weighted version of Calder{\'o}n-Zygmund inequalities \cite{CaZy.57}.
More precisely, let $\alpha$ be a real number such that $1/2<\alpha<3/2$, then
there exists a constant $C>0$, such that for any $u\in\cDd(Z)$~:
\begin{equation}
\label{Calderon}
\nrm{\partial_i\partial_j u}{L_\alpha^2(Z)}\leq C\nrm{\Delta u}{L_\alpha^2(Z)}.
\end{equation}  
}
\end{rmk}

 Using \eqref{Calderon} and thanks to the Closed Range Theorem of Banach, we  prove that~:

\begin{thm}
\label{th_Hm}
Let $\alpha$ be a real number satisfying $1/2<\alpha<3/2$. 
For $m\in\N$, $m\geq3$, the following operator 
\begin{equation}
\label{th_Hm_eq}
\Delta\,:\,H_{\alpha,\#}^m(Z)/\PP_{m-2}'\mapsto H_{\alpha,\#}^{m-2}(Z)/\PP_{m-4}'
\end{equation}
is an isomorphim.
\end{thm}

\noindent The next result is then a straightforward consequence of the latter result.

\begin{thm}
\label{th_Hm_2}
Let $\alpha$ be a real number satisfying $1/2<\alpha<3/2$. 
For $m\in\N$, $m\geq3$, the following operator 
\begin{equation}
\label{th_Hm_2_eq}
\Delta\,:\,H_{\alpha,\#}^m(Z)/\PP_{m-2}^{'\Delta}\mapsto H_{\alpha,\#}^{m-2}(Z)
\end{equation}
is an isomorphim.
\end{thm}

\noindent The next theorem extends Theorem \ref{premiers_isomorphismes}.

\begin{thm}
\label{iso_generales}
Let $\alpha$ be a real number satisfying $1/2<\alpha<3/2$ and
let $\ell\geq1$ be a given integer. Then the Laplace operators defined by
\begin{equation}
\label{iso_generale_1}
\Delta\,:\,H_{-\alpha+\ell,\#}^1(Z)/\pP_{1-\ell}\mapsto H_{-\alpha+\ell,\#}^{-1}(Z)\bot\PP_{-1+\ell}^{'\Delta}
\end{equation} 
and
\begin{equation}
\label{iso_generale_2}
\Delta\,:\,H_{\alpha-\ell,\#}^1(Z)/\PP_{-1+\ell}^{'\Delta}\mapsto H_{\alpha-\ell,\#}^{-1}(Z)\bot\pP_{1-\ell}
\end{equation} 
\noindent are isomorphism.
\end{thm}

\begin{proof}
Observe first that when $\ell=1$, the result is proved in Theorem \ref{premiers_isomorphismes}. 
Observe next that if $m\geq2$ is an integer, the mapping
$$\Delta\,:\,H_{\alpha,\#}^m(Z)/\PP_{m-2}^{'\Delta}\mapsto H_{\alpha,\#}^{m-2}(Z)\bot\pP_{-m+2}$$
is onto. Indeed, if $m=2$, this isomorphism is exactly defined by \eqref{regularite_Hm_iso1} with $\ell=1$.
If $m\geq3$, it is defined by \eqref{th_Hm_2_eq}. Now through duality and transposition, the mapping
$$\Delta\,:\,H_{-\alpha,\#}^{-m+2}(Z)/\pP_{-m+2}\mapsto H_{-\alpha,\#}^{-m}(Z)\bot\PP_{m-2}^{'\Delta}$$ 
is onto. Next, using the same arguments as in the proof of Theorem \ref{regularite_Hm}, we are able to  show that  
for any integer $\ell\geq1$, the operator 
$$\Delta\,:\,H_{-\alpha+\ell,\#}^{-m+2+\ell}(Z)/\pP_{-m+2}\mapsto H_{-\alpha+\ell,\#}^{-m+\ell}(Z)\bot\PP_{m-2}^{'\Delta}$$
\noindent is an isomorphism. Choosing $m=\ell+1$, the operator defined by \eqref{iso_generale_1} is an isomorphism. 
By duality and transposition, the  mapping defined by \eqref{iso_generale_2} is onto as well.\\
\end{proof}

\begin{rmk}
\label{rem_recap}
{\rm
Summarizing theorems \ref{premiers_isomorphismes} and \ref{iso_generales},  we deduce that, for any
$\alpha\in\R$ such that $\alpha\neq\pm(\frac{1}{2}+k)$, $k\in\N^*$, the mapping
\begin{equation}
\label{iso_generale_sans_critique_1}
\Delta\,:\,H_{\alpha,\#}^1(Z)/\PP_{[1/2-\alpha]}^{'\Delta}\mapsto H_{\alpha,\#}^{-1}(Z)\bot\PP_{[1/2+\alpha]}^{'\Delta}
\end{equation} 
is an isomorphism . As a consequence, 
for any $m\in\Z$, for any $\alpha\in\R$ such that $\alpha\neq\pm(\frac{1}{2}+k)$, $k\in\N^*$, the mapping 
\begin{equation}
\label{iso_generale_sans_critique_2}
\Delta\,:\,H_{\alpha,\#}^{m+2}(Z)/\PP_{[m+3/2-\alpha]}^{'\Delta}\mapsto H_{\alpha,\#}^{m}(Z)\bot\PP_{[-m-1/2+\alpha]}^{'\Delta}
\end{equation} 
is an isomorphism.}
\end{rmk}

\section{Isomorphism results for the critical cases}\label{sec.critic}

\noindent Note that the isomorphism result \eqref{iso_generale_sans_critique_2}
is not valid for $\alpha=\pm(\frac{1}{2}+k)$, $k\in\N^*$. In order to extend it to these cases, 
we redefine the weighted spaces. Therefore, proceeding as in \cite{These_Giroire}, for $m\in\Z$, $p\in\N$,
and $\alpha\in\{-\frac{1}{2},\frac{1}{2}\}$, we introduce the space 
\begin{align}
X_{\alpha+p,\#}^{m+p}(Z):=\{&u\in H_{\alpha,\#}^m(Z),\,\,\,\forall\lambda\in\N,\,\,\,1\leq\lambda\leq p,\nonumber\\
&y_2^\lambda u\in H_{\alpha,\#}^{m+\lambda}(Z),\,\,\,u\in H_{\loc,\#}^{m+p}(Z)\},\nonumber
\end{align}
where $H_{\loc,\#}^{m+p}(Z)$ stands for the space of functions that belong to $H_{\#}^{m+p}(Z)$ with compact support 
in the $y_2$ direction. The space $X_{\alpha+p,\#}^{m+p}(Z)$ is a Banach space when endowed with the norm
$$\|u\|_{X_{\alpha+p,\#}^{m+p}(Z)}=\left(\sum_{0\leq\lambda\leq p}\|y_2^\lambda u\|_{H_{\alpha,\#}^{m+\lambda}(Z)}^p
+\|u\|_{H_{\#}^{m+p}(Z^1)}^p\right)^{1/p},$$
where $Z^1:=]0,1[\times]-1,1[$.

\begin{rmk}
{\rm
We  restricted the definition of the space $X_{\alpha+p,\#}^{m+p}(Z)$ to $\alpha\in\{-\frac{1}{2},\frac{1}{2}\}$,
but this definition holds for all $\alpha$ in $\R$. For $\alpha\in\R\setminus\frac{1}{2}\Z$,
since no logarithmic weights appear in the definition of $H_{\alpha+p,\#}^{m+p}(Z)$, 
the spaces $X_{\alpha+p,\#}^{m+p}(Z)$ and $H_{\alpha+p,\#}^{m+p}(Z)$ coincide
algebraically and topologically \cite{These_Giroire}.
}
\end{rmk}

\noindent The next proposition is dedicated to density of $\cD_\#(Z)$ in $X_{\alpha+p,\#}^{m+p}(Z)$. 
In order to prove it, we establish two lemmas giving additional properties on the space $H_{\alpha,\#}^m(Z)$.
The first lemma characterizes the dual space $H_{-\alpha,\#}^{-m}(Z)$.     

\begin{lemma}
\label{lem_caracterisation_duaux}
Let $\alpha\in\R$ and $f\in H_{-\alpha,\#}^{-m}(Z)$. Then for any $v\in H_{\alpha,\#}^m(Z)$, we have 
$$\<f,v\>_{H_{-\alpha,\#}^{-m}(Z)\times H_{\alpha,\#}^m(Z)}=\sum_{0\leq|\mu|\leq m}\int_Z g_\mu\,\partial^\mu v\,d\y,$$
where $g_m\in L^2_{-\alpha}(Z)$ and for any $0\leq|\mu|<m$,
\begin{itemize}
\item $g_\mu\in L_{-\alpha+m-|\mu|}^2(Z)$ if $\alpha\notin\{\frac{1}{2},...,m-\frac{1}{2}\}$,
\item $(\ln(1+\rho^2))^{1/2}g_\mu\in L_{-\alpha+m-|\mu|}^2(Z)$ if $\alpha\in\{\frac{1}{2},...,m-\frac{1}{2}\}$.
\end{itemize}
\end{lemma}

\begin{proof}
We shall only prove the statement for $\alpha\notin\{\frac{1}{2},...,m-\frac{1}{2}\}$.
The proof corresponding to  critical values of $\alpha$ is similar.
We set $\textit{\textbf{E}}=\underset{0\leq|\mu|\leq m}\prod L_{-\alpha+m-|\mu|}^2(Z)$ 
such that for any $\boldsymbol{\psi}=(\psi_\mu)_{0\leq|\mu|\leq m}\in \textit{\textbf{E}}$, 
$$\|\boldsymbol{\psi}\|_{\textit{\textbf{E}}}=\sum_{0\leq|\mu|\leq m}\|\psi_\mu\|_{L_{-\alpha+m-|\mu|}^2(Z)}.$$
The operator $T$ defined by $T:v\in H_{\alpha,\#}^m(Z)\mapsto(\partial^\mu v)_{0\leq|\mu|\leq m}\in\textit{\textbf{E}}$ is isometric.
We now set $\textit{\textbf{G}}:=T(H_{\alpha,\#}^m(Z))$ and $S:=T^{-1}:\textit{\textbf{G}}\mapsto H_{\alpha,\#}^m(Z)$. 
The mapping $L:\textit{\textbf{h}}\in \textit{\textbf{G}}\mapsto\<f,S\textit{\textbf{h}}\>_{H_{-\alpha,\#}^{-m}(Z)\times H_{\alpha,\#}^m(Z)}$ is linear and continuous. 
Therefore thanks to the Hahn-Banach theorem, 
there exists  $\tilde{L}$, a linear and continuous extension of $L$ on \textit{\textbf{E}} such that 
$\|\tilde{L}\|_{\textit{\textbf{E}}'}=\|f\|_{H_{-\alpha,\#}^{-m}(Z)}$. Thanks to the Riesz representation theorem, there exists 
$g_\mu\in L_{-\alpha+m-|\mu|}^2(Z)$ for any $0\leq|\mu|\leq m$, such that 
$$\forall v\in H_{\alpha,\#}^m(Z),\hspace{2mm}\<\tilde{L},v\>
=\sum_{0\leq|\mu|\leq m}\int_Z g_\mu\,\partial^\mu v\,d\y.$$    
\end{proof}

\noindent The second lemma concerns the density of functions with compact support.

\begin{lemma}
\label{lem_convergence}
Let $m\in\Z$, $\lambda\in\N$, $\alpha\in\R$, $\Phi_\ell$ be defined by \eqref{def_phi} and $u\in H_{\alpha,\#}^{m+\lambda}(Z)$.
We set $u_\ell(\y)=\Phi_\ell(y_2)u(\y)$ \textit{a.e.} $\y\in Z$. Then we have 
\begin{equation}
\label{lem_convergence_eq}
\lim_{\ell\to\infty}\|u-u_\ell\|_{H_{\alpha,\#}^{m+\lambda}(Z)}=0.
\end{equation}
\end{lemma}

\begin{proof}
It is clear that for the case $m+\lambda\geq0$, \eqref{lem_convergence_eq} is satisfied.
We shall now focus on the case $m+\lambda<0$. We have 
$$\|u-u_\ell\|_{H_{\alpha,\#}^{m+\lambda}(Z)}=\sup_{\begin{subarray}{l}v\in H_{-\alpha,\#}^{-m-\lambda}(Z)\\ v\neq0\end{subarray}}\frac{|\<u-u_\ell,v\>|}{\|v\|},$$
where $\<\cdot,\cdot\>$ is the duality pairing between $H_{\alpha,\#}^{m+\lambda}(Z)$ and $H_{-\alpha,\#}^{-m-\lambda}(Z)$. 
Besides, thanks to Lemma \ref{lem_caracterisation_duaux} there exist $g_\mu$ such that
$g_{-m-\lambda}\in L_\alpha^2(Z)$ and for any $0\leq|\mu|<-m-\lambda$,
\begin{itemize} 
\item $g_\mu\in L_{\alpha-m-\lambda-|\mu|}^2(Z)$ if $-\alpha\notin\{1/2,...,-m-\lambda-1/2\}$
\item $(\ln(1+\rho^2))^{1/2}g_\mu\in L_{\alpha-m-\lambda-|\mu|}^2(Z)$ if $-\alpha\in\{1/2,...,-m-\lambda-1/2\}$
\end{itemize}
satisfying for any $v\in H_{-\alpha,\#}^{-m-\lambda}(Z)$, $$\<u,v\>
=\sum_{0\leq|\mu|\leq-m-\lambda}\int_Z g_\mu\,\partial^\mu v\,d\y.$$
Therefore, we can write 
$$\<u-u_\ell,v\>=\sum_{0\leq|\mu|\leq-m-\lambda}\int_Z g_\mu(\partial^\mu v-\partial^\mu(\Phi_\ell\, v))d\y.$$
It follows that 
\begin{align}
|\<u-u_\ell,v\>|\leq &\sum_{0\leq|\mu|\leq-m-\lambda}\int_Z|g_\mu(1-\Phi_\ell)\partial^\mu v|d\y\nonumber+C\sum_{0\leq|\mu|\leq-m-\lambda}\,\sum_{0<|k|\leq|\mu|}\int_Z g_\mu\partial^k\Phi_\ell\,\partial^{\mu-k}v d\y.\nonumber
\end{align}
We deduce that
\begin{align}
|\<u-u_\ell,v\>|\leq&\left(\sum_{0\leq|\mu|\leq-m-\lambda}\|g_\mu(1-\Phi_\ell)\|_{L_{\alpha-m-\lambda-|\mu|}^2(Z)}\right.\nonumber\\
&\left.+C\sum_{0\leq|\mu|\leq-m-\lambda}\,\sum_{0<|k|\leq|\mu|}\|g_\mu\,\partial^k\Phi_\ell\|_{L_{\alpha-m-\lambda-|\mu|}^2(Z)}\right)\|v\|_{H_{-\alpha,\#}^{-m-\lambda}(Z)}.\nonumber
\end{align}
The first two terms in the right hand side tend to zero as $\ell$ tends to infinity.
This proves \eqref{lem_convergence_eq} for $m+\lambda<0$.
\end{proof}

\noindent We proceed with three results showing  some properties of the spaces $X_{\alpha,\#}^{m,p}(Z)$.

\begin{proposition}
\label{prop_densite_critique}
Let $m\in\Z$, $p\in\N$ and $\alpha\in\{-\frac{1}{2},\frac{1}{2}\}$, then 
the space $\cD_\#(Z)$ is dense in $X_{\alpha+p,\#}^{m+p}(Z)$.
\end{proposition}

\begin{proof}
Let $u$ be in $X_{\alpha+p,\#}^{m+p}(Z)$. Then according to Lemma \ref{lem_convergence}, $u$ can be approximated by 
$u_\ell\in X_{\alpha+p,\#}^{m+p}(Z)$ with compact support in the $y_2$ direction. 
Therefore $u_\ell$ belongs to $H_{\alpha,\#}^{m+p}(Z)$, and thanks to Proposition \ref{prop_densite}, $u_\ell$
can be approximated by $\psi_\ell\in\cD_\#(Z)$. This ends the proof. 
\end{proof} 

\noindent For $m\in\Z$, $p\in\N$ and $\alpha\in\{-\frac{1}{2},\frac{1}{2}\}$, 
we also define the space $X_{-\alpha-p,\#}^{m-p}(Z)$ that is the dual space of $X_{\alpha+p,\#}^{-m+p}(Z)$.
It is therefore a subspace of $\cDd'(Z)$. 
\begin{proposition}
\label{critique_inclusions}
Let $m,p\in\Z$ and $\alpha\in\{-\frac{1}{2},\frac{1}{2}\}$. 
$X_{\alpha+p+1,\#}^{m+p+1}(Z)$ is densely embedded in $X_{\alpha+p,\#}^{m+p}(Z)$.
\end{proposition} 

\begin{proof}
If $p\in\N$, due to the definition of the space $X_{\alpha+p+1,\#}^{m+p+1}(Z)$, 
the imbedding is straightforward. 
Moreover the density of $\cD_\#(Z)$ in $X_{\alpha+p,\#}^{m+p}(Z)$ 
implies the density of $X_{\alpha+p+1,\#}^{m+p+1}(Z)$ in $X_{\alpha+p,\#}^{m+p}(Z)$.
Hence, if  $p\leq0$, the dense imbedding  holds~: 
$$X_{-\alpha-p,\#}^{-m-p}(Z)\subset X_{-\alpha-p-1,\#}^{-m-p-1}(Z).$$
By duality, the desired result follows.
\end{proof}

\begin{proposition}
\label{critique_derive}
Let $m,p\in\Z$ and $\alpha\in\{-\frac{1}{2},\frac{1}{2}\}$. If
$u$ belongs to $X_{\alpha+p,\#}^{m+p}(Z)$, then 
$\partial_j u$ belongs to $X_{\alpha+p,\#}^{m-1+p}(Z)$, for any $j=1,2$.
\end{proposition}

\begin{proof}
(i) Assume that $p\geq0$. If $u\in X_{\alpha+p,\#}^{m+p}(Z)$, then $u\in H_{\alpha,\#}^m(Z)$
and $\partial_j u\in H_{\alpha,\#}^{m-1}(Z)$. Besides, for any $\lambda\in\N$, such that $1\leq\lambda\leq p$, we have
$y_2^\lambda(\partial_j u)=\partial_j(y_2^\lambda u)-\partial_j(y_2^\lambda) u$. Since $y_2^\lambda u\in H_{\alpha,\#}^{m+\lambda}(Z)$,
it follows that $\partial_j(y_2^\lambda u)\in H_{\alpha,\#}^{m-1+\lambda}(Z)$. Thus it is clear that 
$y_2^\lambda(\partial_1 u)$ belongs to $H_{\alpha,\#}^{m-1+\lambda}(Z)$.
Furthermore, we have $\partial_2(y_2^\lambda)u=\lambda y_2^{\lambda-1}u$ and since $0\leq\lambda-1\leq p-1$,
 $y_2^{\lambda-1}u$ belongs to $H_{\alpha,\#}^{m+\lambda-1}(Z)$. We deduce that for any $\lambda\in\N$, such that $1\leq\lambda\leq p$,
$y_2^\lambda(\partial_j u)$ belongs $H_{\alpha,\#}^{m-1+\lambda}(Z)$. Finally it is straightforward that if
$u$ is in $H_{\#}^{m+p}(Z^1)$, then $\partial_j u$ is in $H_{\#}^{m+p-1}(Z^1)$.\\
(ii) Let now $p<0$. If $u\in X_{\alpha+p,\#}^{m+p}(Z)$, then for any $\varphi\in\cDd(Z)$, we can write 
$$\<\partial_j u,\varphi\>_{\cDd'(Z)\times\cDd(Z)}=-\<u,\partial_j\varphi\>_{\cDd'(Z)\times\cDd(Z)}.$$
Thanks to (i), for any $\varphi\in X_{-\alpha-p,\#}^{-m+1-p}(Z)$, $\partial_j\varphi$ belongs to $X_{-\alpha-p,\#}^{-m-p}(Z)$
which ends the proof.
\end{proof}

\begin{proposition}
\label{critique_y2}
Let $m, p\in\Z$ and $\alpha\in\{-\frac{1}{2},\frac{1}{2}\}$. If $u$ belongs to 
$X_{\alpha+p,\#}^{m+p}(Z)$, then $y_2 u$ belongs to $X_{\alpha+p-1,\#}^{m+p}(Z)\equiv X_{\alpha+p-1,\#}^{m+1+p-1}(Z)$.
\end{proposition}

\begin{proof}
(i) Assume that $p\geq1$. If $u$ belongs to $X_{\alpha+p,\#}^{m+p}(Z)$, then $u$ belongs to $H_{\alpha,\#}^{m}(Z)$
and $y_2 u$ belongs to $H_{\alpha,\#}^{m+1}(Z)$. Next, for any $\lambda\in\N$, such that $1\leq\lambda\leq p-1$, 
we have $y_2^\lambda(y_2 u)=y_2^{\lambda+1}u\in H_{\alpha,\#}^{m+\lambda+1}(Z)$.
Finally if $u$ belongs to $H_{\#}^{m+p}(Z^1)$, then $y_2 u$ also belongs to $H_{\#}^{m+p}(Z^1)$.
We thus deduce that $y_2 u\in X_{\alpha+p-1,\#}^{m+p}(Z)$.\\
(ii) Assume now $p\leq0$. For any $\varphi\in\cDd(Z)$, we can write 
$$\<y_2 u,\varphi\>_{\cDd'(Z)\times\cDd(Z)}=\<u,y_2\varphi\>_{\cDd'(Z)\times\cDd(Z)}.$$
Thanks to (i), for any $\varphi\in X_{-\alpha-p+1,\#}^{-m-p}(Z)\equiv X_{-\alpha-p+1,\#}^{-m-1-p+1}(Z)$, 
$y_2\varphi$ belongs to $X_{-\alpha-p,\#}^{-m-p}(Z)$ which ends the proof.    
\end{proof}

\noindent The above properties allow to extend Theorem \ref{regularite_Hm}.

\begin{thm}
\label{th_iso_critique}
Let $\alpha\in\{-\frac{1}{2},\frac{1}{2}\}$ and $\ell\in\N$. Then the  operator defined by
\begin{equation}
\label{th_iso_critique_eq}
\Delta\,:\,X_{\alpha+\ell,\#}^{1+\ell}(Z)/\R\mapsto X_{\alpha+\ell,\#}^{-1+\ell}(Z)\bot\R
\end{equation}
is an isomorphism.
\end{thm}

\begin{proof}
Observe first that since $X_{\alpha,\#}^1(Z)=H_{\alpha,\#}^1(Z)$ 
and $X_{\alpha,\#}^{-1}(Z)=H_{\alpha,\#}^{-1}(Z)$, the case $\ell=0$ is proved in Theorem
\ref{premiers_isomorphismes}. 
Now we proceed by induction, as in the proof of Theorem \ref{regularite_Hm}.
Assume that the  statement is true for $\ell=k$ and let us prove that it is then true for
$\ell=k+1$. The Laplace operator defined by \eqref{th_iso_critique_eq} is clearly linear and continuous.
It is also injective~:  if $u$ belongs to $X_{\alpha+k+1,\#}^{2+k}(Z)\subset H_{\alpha,\#}^1(Z)$ and 
is harmonic then it is a constant. Let us prove that it is onto. 
Let $f\in X_{\alpha+k+1,\#}^{k}(Z)\bot\R$. Then $f\in X_{\alpha+k,\#}^{-1+k}(Z)\bot\R$ and the induction
assumption implies that there exists $u\in X_{\alpha+k,\#}^{1+k}(Z)$ such that $\Delta u=f$. We 
now consider the equality 
\begin{equation}
\label{th_iso_critique_eq2}
\Delta(y_2\partial_i u)=y_2\partial_i f+2\partial_2\partial_i u.
\end{equation}
Using Proposition \ref{critique_derive} and Proposition \ref{critique_y2}, we deduce that 
$\Delta(y_2\partial_i u)$ belongs to $X_{\alpha+k,\#}^{-1+k}(Z)\subset X_{\alpha+k-1,\#}^{-2+k}(Z)$.
Moreover, thanks  to Propositions \ref{critique_derive} and \ref{critique_y2}, 
if $u\in X_{\alpha+k,\#}^{1+k}(Z)$ then $y_2\partial_i u\in X_{\alpha+k-1,\#}^k(Z)$
and for any $\varphi\in X_{-\alpha-k+1,\#}^{2-k}(Z)$, we have 
$$\<\Delta(y_2\partial_i u),\varphi\>_{X_{\alpha+k-1,\#}^{-2+k}(Z)\times X_{-\alpha-k+1,\#}^{2-k}(Z)}
=\<y_2\partial_i u,\Delta\varphi\>_{X_{\alpha+k-1,\#}^{k}(Z)\times X_{-\alpha-k+1,\#}^{-k}(Z)}.$$ 
Since $\R\subset X_{-\alpha-k+1,\#}^{2-k}(Z)$, we shall  take $\varphi\in\R$. This shows that 
$\Delta(y_2\partial_i u)\in X_{\alpha+k,\#}^{-1+k}(Z)\bot\R$. 
Thanks to the induction assumption, there exists $v\in X_{\alpha+k,\#}^{1+k}(Z)$ satisfying $\Delta v=\Delta(y_2\partial_i u)$. 
Hence $v-y_2\partial_i u\in\R$ and since $\R\subset X_{\alpha+k,\#}^{1+k}(Z)$, it follows that
$y_2\partial_i u\in X_{\alpha+k,\#}^{1+k}(Z)$. Therefore summarizing, we obtained that $u\in H_{\alpha,\#}^1(Z)$,
for any $1\leq\lambda\leq k+1$ $y_2 u\in H_{\alpha,\#}^{1+\lambda}(Z)$. It remains to prove that $u\in H_{\loc,\#}^{k+2}(Z)$.
Let $\chi\in\D(\R)$ and consider the function $\y\mapsto\chi(y_2)u(\y)$ for almost every $\y\in Z$. Then
$\Delta(\chi u)=\chi f+2\nabla\chi\nabla u+u\Delta\chi$ belongs to $H_{\loc,\#}^k(Z)$. It follows from 
standard inner elliptic regularity results that $\chi u$ belongs to $H_{\loc,\#}^{2+k}(Z)$  
 which in turn yields $u\in H_{\loc,\#}^{2+k}(Z)$.
\end{proof}

\begin{rmk}
{\rm
The above theorem extends the Calder{\'o}n-Zygmund inequality \eqref{Calderon} to $\alpha\in\{\frac{1}{2},\frac{3}{2}\}$.
}
\end{rmk}

\noindent Proceeding as in the non-critical cases, we can extend Theorems \ref{th_Hm}, \ref{th_Hm_2}
and \ref{iso_generales}. More precisely, we have~:

\begin{thm}
\label{th_Hm_critique}
Let $\alpha$ be a real number satisfying $1/2\leq\alpha\leq3/2$. 
For $m\in\N$, $m\geq3$, the mapping 
\begin{equation}
\label{th_Hm_eq_critique}
\Delta\,:\,H_{\alpha,\#}^m(Z)/\PP_{m-2}'\mapsto H_{\alpha,\#}^{m-2}(Z)/\PP_{m-4}'
\end{equation}
is an isomorphism.
\end{thm}

\begin{thm}
\label{th_Hm_2_critique}
Let $\alpha$ be a real number satisfying $1/2\leq\alpha\leq3/2$.
For $m\in\N$, $m\geq3$, the mapping 
\begin{equation}
\label{th_Hm_2_eq_critique}
\Delta\,:\,H_{\alpha,\#}^m(Z)/\PP_{m-2}^{'\Delta}\mapsto H_{\alpha,\#}^{m-2}(Z)
\end{equation}
is an isomorphism.
\end{thm}

\begin{thm}
\label{iso_generales_critique}
Let $\alpha\in\{\frac{1}{2},\frac{3}{2}\}$ and
let $\ell\geq1$ be an integer. Then the Laplace operators defined by
\begin{equation}
\label{iso_generale_1_critique}
\Delta\,:\,X_{-\alpha+\ell,\#}^1(Z)/\pP_{1-\ell}\mapsto X_{-\alpha+\ell,\#}^{-1}(Z)\bot\PP_{-1+\ell}^{'\Delta}
\end{equation} 
and
\begin{equation}
\label{iso_generale_2_critique}
\Delta\,:\,X_{\alpha-\ell,\#}^1(Z)/\PP_{-1+\ell}^{'\Delta}\mapsto X_{\alpha-\ell,\#}^{-1}(Z)\bot\pP_{1-\ell}
\end{equation} 
\noindent are isomorphisms.
\end{thm}

\begin{rmk}
{\rm 
At this step we are able to state Thereom \ref{thm.isom.final} that we recall here.
 Let $m\in\Z$ and $\alpha\in\R$, then the Laplace operator  
defined by
\begin{equation}
\label{iso_final}
\Delta\,:X_{\alpha,\#}^{m+2}(Z)/\PP_{q(m+2,\alpha)}^{'\Delta}\mapsto X_{\alpha,\#}^m(Z)\bot\PP_{q(-m,-\alpha)}^{'\Delta}
\end{equation}
is an isomorphism.
}
\end{rmk}

\noindent  In order to prove our main result stated in Theorem \ref{thm.final}, let us first note that we might compare the spaces 
 $X_{\alpha+p,\#}^{m+p}(Z)$ and $H_{\alpha+p,\#}^{m+p}(Z)$ for $\alpha\in\{-\frac{1}{2},\frac{1}{2}\}$.
\begin{proposition}
\label{prop_identite_espaces_critiques}
Let $m,p$ be two integers and let $\alpha\in\{-\frac{1}{2},\frac{1}{2}\}$.
Then 
\begin{equation}
X_{\alpha+p,\#}^{m+p}(Z)\subset H_{\alpha+p,\#}^{m+p}(Z).
\end{equation}
\label{prop_identite_espaces_critiques_inclus}
If, moreover $\alpha=\frac{1}{2}$, then 
\begin{equation}
\label{prop_identite_espaces_critiques_egal}
X_{\frac{1}{2}+p,\#}^{m+p}(Z)=H_{\frac{1}{2}+p,\#}^{m+p}(Z).
\end{equation}
\end{proposition}
\noindent The proof  is based on a partition of unity in the $y_2$ direction and direct computations. 
Note that \eqref{prop_identite_espaces_critiques_egal} is not valid for $\alpha=-\frac{1}{2}$ and $p\geq1$ since the space 
$H_{-1/2+p,\#}^{m+p}(Z)$ is not included in $H_{-1/2,\#}^{m}(Z)$ (see \eqref{inclusions_poids_log}).

\begin{proof}[Sketch of the proof of Theorem \ref{thm.final}]
We shall decompose it into three successive steps:
\begin{enumerate}
\item  If $\alpha>-1/2$ and $f\in L_{\alpha+1}^2(Z)\bot\PP_{[1/2+\alpha]}^{'\Delta}$,
then $G * f\in X_{\alpha+1,\#}^2(Z)$ is a solution of the Laplace equation \eqref{Laplace} 
unique up to a polynomial of $\pP_{q(2,\alpha+1)}$.
Moreover, we have the estimate
$$\|G * f\|_{X_{\alpha+1,\#}^2(Z)/\pP_{q(2,\alpha+1)}}\leq C \|f\|_{L_{\alpha+1}^2(Z)}.$$
This result is a straightforward consequence of the isomorphism result \eqref{iso_final} for $m=0$, Theorem \ref{thm.convo.reg} and Proposition \ref{prop_noyau}. 
\item  If $\alpha\in\R$ and $f\in H_{\alpha,\#}^{-1}(Z)\bot\PP_{q(1,-\alpha)}^{'\Delta}$,
then $G * f\in H_{\alpha,\#}^1(Z)$ is a solution of the Laplace equation \eqref{Laplace} unique up to a polynomial of $\pP_{q(1,\alpha)}$.
Moreover, we have the estimate
$$\|G * f\|_{H_{\alpha,\#}^1(Z)}\leq C \|f\|_{H_{\alpha,\#}^{-1}(Z)}.$$
This second claim is obtained using the first step and duality arguments. Then the isomorphism result \eqref{iso_final} for $m=-1$ and Proposition \ref{prop_noyau} allows to 
identify the variational solution with the solution by convolution. 
\item In the third step we first assume $m\geq0$, $\alpha\in\R$ and let $f\in X_{\alpha,\#}^{m}(Z)\bot\pP_{q(-m,-\alpha)}$. 
Then it is clear that $f\in H_{\alpha-m-1,\#}^{-1}(Z)\bot\pP_{q(1,-\alpha+m+1)}$. Note that by definition $\pP_{q(-m,-\alpha)}$ and  $\pP_{q(1,-\alpha+m+1)}$
coincide. On the one hand, from the second step $G*f$ belongs to $\wsd{1}{\alpha-m-1}{Z}/\pP_{q(1,\alpha-m-1)}$. On the other hand, from \eqref{iso_final}, there exists 
a unique $u$ in $\wsx{m+2}{\alpha}{Z}/\pP_{q(m+2,\alpha)}$, a subset of $\wsd{1}{\alpha-m-1}{Z}/\pP_{q(m+2,\alpha)}$. Then using Proposition \ref{prop_noyau}
we have $u-G*f\in\pP_{q(m+2,\alpha)}$.\\
Finally using again a duality argument we prove the statement for $m\leq0$ and $\alpha\in\R$.

\end{enumerate}

%The first claim of the theorem follows from Theorem   \ref{thm.convo.reg} , the isomorphism result  \eqref{iso_generale_sans_critique_2} for $m=0$ and from Proposition \ref{prop_noyau}. The second claim is obtained using the first statement by a duality argument, the  isomorphism \eqref{iso_generale_sans_critique_1} and Proposition \ref{prop_noyau}.
\end{proof}
%\begin{thm}
%Let $f\in L_{3/2}^2(Z)\bot\R$,
%then $G * f\in H_{3/2,\#}^2(Z)$ is a solution of \eqref{Laplace} unique up to a constant.
%Moreover, we have the estimate
%$$\|G * f\|_{H_{3/2,\#}^2(Z)/\R}\leq\|f\|_{L_{3/2}^2(Z)}.$$
%\end{thm}

\section{Conclusion}

This paper is a first attempt towards a systematic analysis of boundary layer problems
in periodic strips. This framework provides generic spaces avoiding tedious 
{\em a priori} definitions of solutions' behavior at infinity, among other advantages.

Nevertheless, in the homogenization theory, solutions of boundary 
layer problems  often  converge exponentially to zero up to some polynomial at infinity 
(when $y_2 \to \infty$) \cite{Ba.Siam.76}. Although in the  framework above this corresponds
to a solution belonging modulo this polynomial to any weighted polynomial space, one would like
to quantify the exponential rate of convergence \cite{Ba.Siam.76}, for instance. 
In a forecoming work, our goal should be to provide 
an adequate framework of these results in the weighted context, and to consider exterior 
periodic domains \cite{AmGiGiII.97} for the Laplace and the Stokes operators.

%\section*{References}
%\input{melin}

%\begin{proof}
%\begin{coro}
%\label{iso_pour_pression}
%Let $\alpha$ be a real number satisfying $0\leq\alpha<1$ and $\alpha\neq1/2$.
%Let $\ell$ be an integer. Then the Laplace operators defined by
%\begin{equation}
%\label{iso_pression_1}
%\Delta\,:\,L_{-\alpha+\ell}^2(Z)\mapsto H_{-\alpha+\ell,\#}^{-2}(Z)\bot\PP_{[3/2-\alpha+\ell]}^{'\Delta}
%\end{equation} 
%and
%\begin{equation}
%\label{iso_pression_2}
%\Delta\,:\,L_{\alpha-\ell}^2(Z)/\PP_{[\ell-1/2-\alpha]}^{'\Delta}\mapsto H_{\alpha-\ell,\#}^{-2}(Z)\bot\PP_{[3/2+\alpha-\ell]}'
%\end{equation} 
%\noindent are isomorphism.
%\end{coro}

%\section{The Stokes equations in a periodic infinite strip}

%\vspace{-1cm}

\bibliographystyle{elsarticle-num}
%\begin{NoHyper}
\def\cprime{$'$}

%\end{NoHyper}






%\end{thebibliography}

\end{document}